\documentclass[12pt]{article}

\usepackage{amsmath,amsthm,amsfonts}
\usepackage{graphicx}
\usepackage{authblk}
\usepackage[colorlinks=true]{hyperref}
\hypersetup{urlcolor=blue, citecolor=blue}

\theoremstyle{plain}
\theoremstyle{definition}
\theoremstyle{remark}
\newtheorem{thm}{Theorem}[section]
\newtheorem{lem}[thm]{Lemma}

\newtheorem*{rem}{Remark}

\newcommand{\abs}[1]{\vert#1\vert}

\newcommand{\bq}{\begin{equation}}
\newcommand{\eq}{\end{equation}}

\newcommand{\Rn}{\mathbb{R}^n}

\newcommand{\al}{\alpha}

\newcommand{\FL}{(-\Delta)^{\al/2}}
\newcommand{\FLh}{(-\Delta_h)^{\al/2}}

\newcommand{\Dd}{D}

\newcommand{\Ddh}{\Dd_h}
\newcommand{\DdCh}{\Dd^\mathbf{C}_h}

\newcommand{\GL}{{Gr\"unwald-Letnikov~}}
\newcommand{\Cal}{{C_{1,\al}}}
\newcommand{\sinc}{\operatorname{sinc}}
\newcommand{\Zh}{\mathbb{Z}_h}

\newcommand{\dr}{\,\mathrm{d}}

\numberwithin{figure}{section}
\numberwithin{table}{section}

\title{Finite difference methods for fractional Laplacians}

\author[1]{Yanghong Huang\thanks{yanghong.huang@manchester.ac.uk}}
\author[2]{Adam Oberman\thanks{adam.oberman@mcgill.ca}}
\affil[1]{School of Mathematics, The University of Manchester, Manchester, M13 9PL, UK}
\affil[2]{Department of Mathematics and Statistics, McGill University, Montreal, QC H3A 0B9, Canada}

\begin{document}
\maketitle

%
%
%
%

\begin{abstract}
The fractional Laplacian $(-\Delta)^{\alpha/2}$ is the prototypical non-local
elliptic operator. While analytical theory has been
advanced and understood for some time,
there remain many open problems in the numerical analysis of the operator.
In this article, we study  several different  finite difference
discretisations of the fractional Laplacian on uniform grids in one dimension
that takes the same form. Many properties can be compared and summarised in
this relatively simple setting, to tackle more important questions
like the nonlocality, singularity and flat tails common in practical implementations.
The accuracy and the asymptotic behaviours of the methods are also studied,
together with treatment of the far field boundary conditions,
providing a unified perspective on the further development of the scheme in higher dimensions.
\end{abstract}

\tableofcontents


\section{Introduction} \label{sec:introd}
Anomalous diffusion is a phenomenon of current interest,  occurring in various
physical systems~\cite{hilfer2000applications} as well as in financial
modelling~\cite{MR1809268,MR1937584,metzler2004restaurant}.
In this article, we focus on the numerical approximations of the fractional Laplacian,
the prototypical nonlocal operator intimately related to anomalous diffusion.
In the whole space $\Rn$, the fractional Laplacian of order $\alpha \in (0,2)$ can be
defined in several equivalent ways~\cite{MR0350027,MR0290095,Vald11}.
Two of the most common ones are given by the Fourier transform
\begin{equation}
\label{eq:fftfrac}
\mathcal{F}[\FL u](\xi) = |\xi|^{\alpha}\mathcal{F}[u](\xi).
\end{equation}
and by the singular integral
\begin{equation}\label{eq:rieszfrac}
 \FL u(x) = C_{n,\al} \, \int_{\mathbb{R}^n}\!
\frac{u(x)-u(y)}{|x-y|^{n+\al}} \,\mathrm{d}y,
\end{equation}
where the constant coefficient $C_{n,\al}$ is given by
 \begin{equation}\label{eq:constC}
 C_{n,\al} = \frac{\al 2^{\al-1}\Gamma\left(\frac{\al+n}{2}\right)}
{\pi^{n/2}\Gamma\left(\frac{2-\al}{2}\right)},
\end{equation}
and $\Gamma(t) = \int_0^\infty s^{t-1} e^{-s} \mathrm{d}s$ is the Gamma
function. Using the symmetric combination of two one-sided fractional derivatives~\cite{MR1219954,MR0361633,podlubny1998fractional,MR1347689},
the Fractional Laplacian in one dimension can be written as
\begin{equation}\label{eq:fracdevlap}
(-\Delta)^{\alpha/2}u(x) =
\frac{{}_{-\infty}D_x^\alpha u(x)+{}_xD_\infty^\al u(x)}{2\cos (\alpha
\pi/2)}, \quad \alpha \neq 1.
\end{equation}
Based on tools from functional calculus, the fractional Laplacian (more generally
the fractional power of positive operators) can be represented as
\begin{equation}\label{eq:funcalfrac}
(-\Delta)^{\alpha/2}u = \frac{1}{\Gamma(-\frac{\alpha}{2})}\int_0^\infty
( e^{s\Delta }u - u)s^{-1-\alpha/2}\dr s.
\end{equation}
Over the past decade, many important theoretical results of differential equations with
the fractional Laplacian are established by building upon the definition using
$\alpha$-harmonic extension proposed by Caffarelli and Silvestre~\cite{caffarelli2007extension}.

All the previous definitions are equivalent in the whole space~\cite{kwasnicki2015ten},
at least when the operator is applied to proper classes of functions that decay fast
enough at infinity.  However, on bounded domains there are many non-equivalent definitions,
according to the particular boundary conditions used.
For this reason, the focus of the finite difference methods here is on the whole
space, allowing a unified treatment of viewing many existing schemes
and deriving new schemes.  However, boundary conditions remain an important
consideration, and will be touched on in later sections.

The fractional Laplacian can be evaluated numerically in various ways according
to the equivalent definitions above, for instance, using either the
spectral definition in Fourier space or the singular integral representation.
A vast majority of existing schemes are based on fractional
derivatives~\cite{Meerschaert200465,Tadjeran2006205},
using variants of the classical \GL derivative~\cite{diethelm2005algorithms,MR0361633}.
When used in the discretisation of fractional diffusion
equations, the resulting schemes can usually be interpreted as
a random walk with long range jumps~\cite{gorenflo1998random,gorenflo1999discrete}.
Another popular class of schemes use spectral decomposition in Fourier
or other basis~\cite{MR3292533}, usually only valid for finite spatial domains.
In recent years, more numerical schemes are proposed, based on
either the singular integral~\cite{MR2552219,HuangOberman1}, the one from functional
calculus~\cite{ciaurri2015fractional} and the harmonic extension~\cite{Felix,MR3348172}.

Despite the abundance of existing numerical methods to choose from, the presence of nonlocality,
singularity as well as prevalent flat tails in the far field remains a  challenging problem.
In this paper, we provide a unified framework
for the finite difference approximation of the fractional Laplacian.
In one dimension, let $u = \{u_j\}_{j\in \mathbb{Z}}$ be
a function defined on the uniform grid $\Zh =\{jh \mid j \in \mathbb{Z}\}$ with spacing $h>0$.
The discrete operator, denoted as $(-\Delta_h)^{\alpha/2}$,
evaluated at $x_j=jh$ is given by
\begin{equation}\label{FLh}\tag{FLh}
   (-\Delta_h)^{\alpha/2} u_j = \sum_{k=-\infty}^\infty (u_j-u_{j-k})w_k,
\end{equation}
with some prescribed weights $\{w_k\}_{k \in \mathbb{Z}}$. This special form of the scheme
clearly resembles the singular integral~\eqref{eq:rieszfrac}
and is shown to be a multiplier in the appropriate spectral space below.
This scheme already appeared several times explicitly or implicitly in the
literature~\cite{biswas2010difference,MR2795714,MR2552219,HuangOberman1}, including the
simplest choice $w_k = C|k|^{-\alpha-1} (k\neq 0)$
from~\cite{Vald11}. However, its systematic study will be performed here for the first time, linking
different definitions in the continuous setting and motivating new development in the future.

A few reasonable constraints on the weights $w_k$ can be readily imposed without worrying about
their exact values.
Since the fractional Laplacian is symmetric under reflection,  it is natural to require
that $w_k = w_{-k}$. When the scheme~\eqref{FLh} is used in the discretisation of the simplest
fractional heat equation $u_t + (-\Delta)^{\alpha/2}u=0$ for the probability density $u$,
the explicit Euler scheme $(u_j^{n+1} - u_j^{n})/\Delta t= (-\Delta_h)^{\alpha/2}u_j^n$
for $u^{n+1}_j$ at time $t_{n+1}=(n+1)\Delta t$ and space $x_j=jh$ can be written as
\begin{equation}\label{eq:disfracheat}
u^{n+1}_j = (1+\Delta t w_0)u^{n}_j + \sum_{k\neq 0} \Delta t w_ku_{j-k}^n.
\end{equation}
The associated random walk interpretation as performed in~\cite{gorenflo1998random,gorenflo1999discrete}
suggests that the transition probability $\Delta t w_k$ is non-negative, or the weights
$w_k$ with $k\neq 0$ are non-negative. We will show that
basic properties like  symmetry and non-negativity of the weights are valid
for almost all schemes discussed later, and are essential for other related properties like
discrete maximum principle.

The analysis and numerical experiments with the scheme~\eqref{FLh} performed in this paper
is organised as follows. The semi-discrete Fourier Transform
will be introduced next, revealing the relationship between the weights
in~\eqref{FLh} and the symbol in the spectral space,
followed by some identities and inequalities in Section~\ref{sec:ind}.
Different weights are either derived or reviewed in Section~\ref{sec:weights},
and their properties are summarised and compared in Section~\ref{sec:DCweights}.
Other issues about the treatment of far field boundary conditions
and the convergence of the scheme for equations with the fractional Laplacian operator are discussed in
Section~\ref{sec:truncation} and Section~\ref{sec:conv}.  Finally we end with this paper
with several numerical experiments in Section~\ref{sec:nonlinear} and some final remarks on
the extension into higher dimensions.

\section{Semi-discrete Fourier analysis and spectral representation}
\label{sec:semiFFT}

Since $(-\Delta)^{\alpha/2}$ is a pseudo-differential operator with  symbol
$|\xi|^\alpha$, its numerical discretisation is also expected to have
a symbol (if it exists in an appropriate spectral space) that approximates
$|\xi|^\al$.  We show that \emph{semi-discrete Fourier transform},
already well-known in the numerical analysis community~\cite{MR1776072},
is the right tool to analyse the discrete scheme~\eqref{FLh}.
The one-to-one correspondence between the weights $w_k$ and
the associated symbol will be established: the symbol is defined as
a discrete sum involving the weights, and the weights
can be calculated as a Fourier integral of the symbol.
Moreover, explicit symbols are obtained  for many popular schemes from the literature,
including the spectrally accurate scheme with the exact symbol $|\xi|^\alpha$.
To facilitate the discussion below,
we introduce some function spaces on the grid $\Zh = \{hk \mid k\in\mathbb{Z}\}$ and
on the interval $I_h=[-\pi/h, \pi/h]$,
\begin{align*}
 \ell^2(\Zh)
&= \Big\{v:\Zh\to\mathbb{R}  ~\Big |~
h\sum_{j=-\infty}^\infty |v_j|^2 <\infty\Big\}, \\
L^2\left(I_h \right)
&= \Big\{ \hat{v}:
I_h \to \mathbb{R}  ~\Big |~
\int_{-\pi/h}^{\pi/h}\! |\hat{v}(\xi)|^2 \,\mathrm{d}\xi <\infty
\Big\}.
\end{align*}
Both $\ell^2(\Zh)$ and  $L^2\left(I_h \right)$ are Hilbert spaces,
endowed with the inner products
\begin{equation}\label{eq:innerprod}
\quad
\big\langle u,v\big\rangle_{\ell^2(\mathbb{Z}_h)}
= h\sum_{j=-\infty}^\infty u_j\overline{v_j}\qquad  \mbox{and} \qquad
\big\langle \hat{u},\hat{v}\big\rangle_{L^2(I_h)}
= \int_{-\pi/h}^{\pi/h} \!\hat{u}(\xi)\overline{\hat{v}(\xi)}\,\mathrm{d}\xi.
\end{equation}

\subsection{Background on semi-discrete Fourier analysis}

Semi-discrete Fourier transform is intimately connected to the widely used
Fourier transform and Fourier series, and can be motivated from the Fourier transform
\begin{equation}\label{eq:FourierTransform}
\hat{v}(\xi) = \int_{-\infty}^\infty e^{-i\xi x}v(x)\dr x.
\end{equation}
If the function $v$ is defined on the grid $\Zh$ of our interest here,
the natural modification of~\eqref{eq:FourierTransform}
is to replace the integral by its Trapezoidal rule
\begin{equation}\label{eq:semiFFTa}
\hat{v}(\xi) = \mathcal{F}[v](\xi)= h \sum_{j=-\infty}^\infty e^{-i\xi x_j}v_j,
\end{equation}
which is taken as the definition of our \emph{semi-discrete Fourier Transform}.
As a result, the transformed function $\hat{v}(\xi)$ is periodic
with period $2\pi/h$ and it is more convenient to restrict the spectral space
to be the interval $I_h=[-\pi/h, \pi/h]$ instead of the real line $\mathbb{R}$.
The Fourier transform~\eqref{eq:FourierTransform} can be recovered from~\eqref{eq:semiFFTa}
in the limit as $h$ goes to  zero, and more detailed
exposition about these transforms can be found
in~\cite[Chapter 2]{MR1776072}.

Once the semi-discrete Fourier Transform~\eqref{eq:semiFFTa} is defined, its
inverse transform can also be readily worked out and takes the form
\begin{equation}\label{eq:semiFFTb}
{v}_j = \mathcal{F}^{-1}[\hat{v}](x_j)=
\dfrac{1}{2\pi}\int_{-\pi/h}^{\pi/h} \! e^{i\xi x_j}\hat{v}(\xi)\,\mathrm{d}\xi.
\end{equation}
The pair of transforms~\eqref{eq:semiFFTa} and~\eqref{eq:semiFFTb} are inverse
to each other,  which is equivalent to the  identity
(in the sense of distributions)
\begin{equation}\label{eq:invsemiFFT}
 \frac{h}{2\pi} \sum_{j=-\infty}^\infty
e^{i\xi x_j} = \delta(\xi), \qquad \xi \in (-\pi/h,\pi/h),
\end{equation}
where $\delta(\xi)$ is the Dirac delta function. This pair  can also
be viewed as the reverse of Fourier series~\cite[Chapter 2]{MR1776072},
though here the grid size $h$  appears explicitly in both the physical
domain $\Zh$ and the spectral domain $I_h$.

We now state the well-known convolution theorem,
whose proof is standard in Fourier Analysis.

\begin{lem}[Convolution as a multiplier in the spectral space]\label{lem:cv}
 Let $v$, $w$ be functions in $\ell^2(\Zh)$ and
 their semi-discrete Fourier transform $\hat{v}$, $\hat{w}$ be functions in $L^2(I_h)$.
Define the operator $\mathcal{D}: \ell^2(\Zh) \to
\ell^\infty(\Zh)$ by
\[
(\mathcal{D} v)_j = h\sum_{k=-\infty}^\infty v_{j-k} w_{k},
\]
then the representation of $\mathcal{D}$ in spectral space is
a multiplication of $\hat{v}$ by  $\hat{w}$, that is,
\[
\widehat{\mathcal{D}v}(\xi) = \hat{w}(\xi) \hat{v}(\xi).
\]
\end{lem}

Before establishing the connection between the weights and the symbol, we make the following observation.
Since the precise value of $w_0$ in~\eqref{FLh} is irrelevant (its coefficient
is identically zero), without loss of generality in this paper we set
\begin{equation}\label{w0}
w_0 = -\sum_{k\neq 0}w_k.
\end{equation}
Using this convention, the operator~\eqref{FLh} becomes a discrete convolution
\begin{equation}
\label{FLh1}\tag{FLh'}
   (-\Delta_h)^{\alpha/2}u_j = - \sum_{j=-\infty}^\infty u_{j-k}w_k.
\end{equation}

\subsection{Spectral representation of the
  schemes~\eqref{FLh}}\label{sec:specrep}

Next we obtain the representation of the general scheme~\eqref{FLh} in the
spectral space, as a direct application of Lemma~\ref{lem:cv}.

\begin{lem} Let $u$ and $w$ be functions in $\ell^2(\Zh)$,
    then the general scheme~\eqref{FLh} is a multiplication
    in spectral space, that is
\begin{equation}\label{Duhat}
    \widehat{(-\Delta_h)^{\alpha/2} u}(\xi)=M_h(\xi)\hat{u}(\xi),
\end{equation}
with the symbol $M_h(\xi)=- \hat{w}(\xi)/h$.
\end{lem}

The structure of the symbol $M_h(\xi)$ can be further simplified by
isolating the grid size $h$.
Since $(-\Delta)^{\alpha/2}$ is a spatial derivative of order $\alpha$,
it is reasonable to assume that the only dependence of the weights
$\{w_k\}_{k={-\infty}}^{\infty}$
on $h$ is the factor $h^{-\alpha}$.  Stated differently, $h^\alpha w_k$
does not depend on $h$. As a result, it is more convenient
to consider the following rescaled, $h$-independent symbol
\begin{equation}\label{eq:wtoM}
 M(\xi) = h^{\alpha}M_h(\xi/h)
=-h^{\alpha}\sum_{k=-\infty}^\infty w_k e^{-ik\xi},
\end{equation}
which is defined on the fixed interval $[-\pi,\pi]$. We will focus on
the rescaled symbol $M(\xi)$ instead of $M_h(\xi)$ in the rest of the paper.

On the other hand, if  $M(\xi)$ (or equivalently $M_h$) is known,
the corresponding weights $\{w_k\}_{k={-\infty}}^{\infty}$ can be obtained
using the inverse transform~\eqref{eq:semiFFTb}, that is,
\begin{equation}\label{eq:SpectralWeight}
\qquad w_k =
-\frac{h}{2\pi}\int_{-\pi/h}^{\pi/h}\! M_h(\xi) e^{i\xi x_k}\,\mathrm{d}\xi
=-\frac{h^{-\alpha}}{\pi}\!
\int_0^\pi M(\xi) \cos ( k\xi)  \,\mathrm{d}\xi, \quad
k \neq 0,
\end{equation}
where the symmetry $M(-\xi)=M(\xi)$, as a result of the fact $w_k=w_{-k}$, is used.
Therefore, one can develop finite difference schemes of the form~\eqref{FLh}
by proposing appropriate symbols $M(\xi)$, provided that
the oscillatory integral~\eqref{eq:SpectralWeight} defining the weights can
be evaluated accurately, by either reducing to closed form expression or using numerical quadrature.

\begin{rem} If $M(\xi)$ is continuous at the origin and $M(0)=0$, then $w_0$
calculated from~\eqref{eq:SpectralWeight} with $k=0$
satisfies the convention $w_0=-\sum_{k\neq 0}w_k$
automatically. This fact can be verified easily
using the identity~\eqref{eq:invsemiFFT}, since
\[
 \sum_{k=-\infty}^\infty w_k =-\frac{h}{2\pi}\int_{-\pi/h}^{\pi/h}
{M}_h(\xi) \left(\sum_{k=-\infty}^\infty e^{i\xi x_k}
\right)\,\mathrm{d}\xi
=-\int_{-\pi/h}^{\pi/h} {M}_h(\xi)\delta(\xi)\,\mathrm{d}\xi=0.
\]
\end{rem}
Once the one-to-one correspondence
between the weights and the symbol as illustrated in
~\eqref{eq:wtoM} and~\eqref{eq:SpectralWeight} is clear,
we can examine the (rescaled) symbols of existing schemes on one hand and
propose new schemes from the symbols on the other hand.
Before going into these details, we show that many discrete identities
and inequalities can be established as a direct consequence of the special
form the scheme~\eqref{FLh}, which is often independent of the particular
choices of the weights (or the symbols).

\section{Discrete identities and inequalities}\label{sec:ind}

Because of the similar structure between the singular integral~\eqref{eq:rieszfrac}
and the discrete scheme~\eqref{FLh}, many important identities and
inequalities of the continuous operator have their
discrete counterparts. These identities are valid
under mild conditions on the decay of the discrete functions involved,
while the inequalities usually require non-negativity of the
weights $w_k$ with $k \neq 0$.

By the definitions of the inner product~\eqref{eq:innerprod} and the
pair of transforms~\eqref{eq:semiFFTa}
and~\eqref{eq:semiFFTb}, the following Parseval's identity holds.
\begin{lem}[Parseval's identity] Let $u$ and $v$ be two functions
  in $\ell^2(\mathbb{Z}_h)$ and $\hat{u},\hat{v} \in L^2(I_h)$ be
  their semi-discrete Fourier transforms. Then
\[
\langle u, v\rangle_{\ell^2(\mathbb{Z}_h)} = \frac{1}{2\pi}
\langle \hat{u},\hat{v}\rangle_{L^2(I_h)},
\]
and in particular $\|u\|_{\ell^2(\mathbb{Z}_h)}
= \frac{1}{\sqrt{2\pi}}\|\hat{u}\|_{L^2(I_h)}$.
\label{lem:parseval}
\end{lem}

Similar to the continuous fractional Laplacian, the discrete operator
$(-\Delta_h)^{\alpha/2}$ is also self-adjoint.
\begin{lem} Let $u$ and
$v$ be two functions in $\ell^2(\Zh)$  such that $u_{j}$ and $v_{j}$
decay to zero fast enough as $|j|\to\infty$, then
\[
\big\langle (-\Delta_h)^{\alpha/2} u,v\big\rangle_{\ell^2(\Zh)} = \big\langle u,
(-\Delta_h)^{\alpha/2}v\big\rangle_{\ell^2(\Zh)}.
\]
\end{lem}

This lemma can be proved easily using the equivalent definition~\eqref{FLh1} and
the fact that the weights $w_k$ are real and symmetric. The inner product
in this lemma also suggests the \emph{discrete energy}
$\mathcal{E}[u] = \frac{1}{2}
\big\langle (-\Delta_h)^{\alpha/2} u,u\big\rangle_{\ell^2(\Zh)}$, that is,
\begin{equation}\label{eq:esum}
  \mathcal{E}[u]= \frac{h}{4} \sum_{j=-\infty}^\infty \sum_{k=-\infty}^\infty
|u_j-u_{j-k}|^2w_k,
\end{equation}
or equivalently
\begin{equation}\label{eq:eint}
 \mathcal{E}[u] = \frac{1}{4\pi}\int_{-\pi/h}^{\pi/h}
 M_h(\xi)|\hat{u}(\xi)|^2 \dr\xi.
\end{equation}
Therefore, from the expressions~\eqref{eq:esum} and~\eqref{eq:eint},
the energy is non-negative under certain conditions, as summarized in the following lemma.

\begin{lem} If either the symbol $M(\xi)$ or
  the weights $w_k$ ($k\neq 0$) are
  non-negative, then the energy $\mathcal{E}[u]$ is non-negative.
\end{lem}


The following inequalities are useful in various estimates
for the analysis of differential equation where the fractional Laplacian
is discretised with the scheme~\eqref{FLh}, motivated from
their continuous counterparts initiated in the study of the
quasi-geostrophic equation by C{\'o}rdoba and
C{\'o}rdoba~\cite{MR2032097,MR2084005} and then generalized by Ju~\cite{MR2123380}.
\begin{lem} Let $\alpha \in (0,2)$, $p>1$ and $u$, $v$ be functions defined
  on the grid $\mathbb{Z}_h$ such that $v_j=|u_j|^{p}$. If
  the weights $\{w_k\}_{k={-\infty}}^{\infty}$ are non-negative for
  $k\neq 0$, the following point estimate holds
\[
|u_j|^{p-2}u_j(-\Delta_h)^{\alpha/2} u_j \geq \frac{1}{p} (-\Delta_h)^{\alpha/2}
v_j.
\]
\end{lem}
\begin{proof}
 From Young's inequality, for any $u_j$ and $u_{j-k}$,
\[
 \frac{p-1}{p}|u_j|^{p}+\frac{1}{p}|u_{j-k}|^{p}
\geq |u_j|^{p-1} |u_{j-k}| \geq |u_j|^{p-2} u_j u_{j-k},
\]
which is equivalent to
\[
|u_j|^{p-2} u_j \big(u_j-u_{j-k}\big) \geq \frac{1}{p}\big( |u_{j}|^{p}-|u_{j-k}|^{p}\big)
= \frac{1}{p}\big( v_j-v_{j-k}\big) .
\]
Therefore, from the definition of the discrete operator,
\begin{multline*}
 |u_j|^{p-2}u_j(-\Delta_h)^{\alpha/2} u_j
= \sum_{k\neq 0} |u_j|^{p-2} u_j \big(u_j-u_{j-k}\big)w_k \\
\geq \frac{1}{p}\sum_{k\neq 0} \big( |u_{j}|^{p}-|u_{j-k}|^{p}\big)w_k
= \frac{1}{p}(-\Delta_h)^{\alpha/2}v_j.
\end{multline*}
\end{proof}

This lemma can be used to prove the discrete version of the Stroock-Varopoulos
inequality.
\begin{lem} Let $\alpha \in (0,2)$ and $p>2$. If both the function $u$ and
its discrete fractional Laplacian $(-\Delta_h)^{\alpha/2}u$ are defined on $\mathbb{Z}_h$,
\[
\big\langle |u|^{p-2}u,(-\Delta_h)^{\alpha/2}u\big\rangle_{\ell^2(\Zh)}  \geq \frac{2}{p}
\big\langle |u|^{p/2}, (-\Delta_h)^{\alpha/2}|u|^{p/2}\big\rangle_{\ell^2(\Zh)} .
\]
\end{lem}
\begin{proof}
Similarly by Young's inequality,
\[
\frac{p-2}{p}|u_j|^{\frac{p}{2}}+\frac{2}{p}|u_{j-k}|^{\frac{p}{2}}
\geq |u_j|^{\frac{p}{2}-1}|u_{j-k}|
\geq |u_j|^{\frac{p}{2}-2}u_ju_{j-k},
\]
which can be rearranged as
\[
|u_j|^p-|u_j|^{p-2}u_ju_{j-k} \geq \frac{2}{p}|u_j|^{\frac{p}{2}}
\big(|u_j|^{\frac{p}{2}} - |u_{j-k}|^{\frac{p}{2}}\big).
\]
Therefore,
\begin{multline*}
 \big\langle |u|^{p-2}u,(-\Delta_h)^{\alpha/2}u\big\rangle_{\ell^2(\Zh)}
=h\sum_j \sum_{k\neq 0} |u_j|^{p-2}u_j(u_j-u_{j-k})w_k \\
\geq \frac{2h}{p}\sum_j \sum_{k\neq 0}
|u_j|^{\frac{p}{2}}\big(|u_j|^{\frac{p}{2}} - |u_{j-k}|^{\frac{p}{2}}\big)w_k
=\frac{2}{p}
\big\langle |u|^{p/2}, (-\Delta_h)^{\alpha/2}|u|^{p/2}\big\rangle_{\ell^2(\Zh)}.
\end{multline*}
\end{proof}

When the scheme~\eqref{FLh} is used to discretise the fractional Laplacian
operator in differential equations, above identities or inequalities
are critical to establish similar energy estimates and other qualitative properties
as in the continuous settings. If the original differential equation is derived
from variational principles, the discretised equation can also be obtained from
discrete variational principles involving the energy $\mathcal{E}$. Consequently,
the analogy between the continuous operator and its discretization~\eqref{FLh}
allows the development of a parallel theory or method in both settings. However,
we will focus on practical  numerical aspects below and leave
further investigations related to these discrete estimates as future work.

\section{Presentation of the weights}\label{sec:weights}
In this section we present several finite difference
approximations of the fractional Laplacian operator
of the form~\eqref{FLh}, including  schemes constructed
from explicit symbols like $|\xi|^{\alpha}$ or
$( 2- 2\cos\xi)^{\alpha/2}$, a variant of the \GL method associated with fractional derivatives,
and quadrature-difference methods based on the singular
integral~\eqref{eq:rieszfrac}. The focus of this section is
on the derivation of the weights or the associated symbols,
while other important properties are summarised and compared in the next section.

\begin{rem} The expressions of the weights $w_k$ for different schemes are usually
only  given for $k>0$, with the assumption that $w_{k}=w_{-k}$ for $k<0$ and
$w_0=-\sum_{k\neq 0}w_k$, although some formulas are still valid for
negative indices $k$. When the weights or the symbols are referred
for a particular method, superscripts are used to label and distinguish them.
For instance, $w_k^{SP}$ and $M^{SP}(\xi)$ are the weights and the associated
rescaled symbol for the spectrally accurate scheme discussed in the next subsection.
\end{rem}

\subsection{Spectral weights and sinc interpolation}

Spectral methods are widely used in scientific computing~\cite{MR1874071,MR2867779, MR1776072},
where the functions involved are expanded using basis defined on the whole computational domain.
Classical derivatives are usually evaluated with Fast Fourier Transform (FFT), and the error
decays exponentially fast for smooth functions, hence so call the \emph{spectral accuracy}.
For functions define on the whole space with fast decay at infinity, similar error properties
are preserved when the computational domain is truncated to be finite.

The fractional Laplacian operator, as a multiplier in the spectral space,
can be evaluated using FFT-based spectral methods, which is still spectrally accurate
for smooth \emph{periodic} functions. However, if the function $u$ is defined on the whole real line,
the nonlocality of the operator has a non-trivial effect: when the computational domain
is truncated to be $[-L,L]$ for some finite $L>0$, the error of the conventional Fourier
spectral method decays only algebraically (on the domain size $L$), even $u$
approaches zero fast at infinity. In fact, the result returned from FFT is
$\FL \tilde{u}(x)$ instead of $\FL u(x)$, where $\tilde{u}$ is the period
extension of $u$ from the interval $[-L,L]$ to the whole real line, i.e.,
\[
 \tilde{u}(x) = \sum_{j=-\infty}^\infty u(x-2jL).
\]
More precisely, if $u$ is essentially supported on the interval $[-L,L]$, the leading order error is
\[
 \FL\tilde{u}(x) - \FL u(x)=-C_{1,\alpha}\sum_{j\neq 0} \int_{\mathbb{R}}\!
 \frac{u(x-2jL)}{|x-y|^{1+\alpha}}\,\mathrm{d}y = O(L^{-1-\alpha}),
\]
for $x \in (-L,L)$.
This algebraically slow convergence is shown in Figure~\ref{fig:fft}(a) for $u(x)=e^{-x^2}$,
whose fractional Laplacian at the origin can be calculated explicitly as
\[
 (-\Delta)^{\alpha/2}u(0) = \frac{1}{\sqrt{\pi}}\int_0^\infty\! k^\alpha e^{-k^2/4}\,
 \mathrm{d}k = 2^\alpha\Gamma\Big(\frac{1+\alpha}{2}\Big)\big/\sqrt{\pi}.
\]
The leading $O(L^{-1-\alpha})$ error at the origin is also verified numerically
in Figure~\ref{fig:fft}(b).
\begin{figure}[htp]
 \begin{center}
 \includegraphics[totalheight=0.27\textheight]{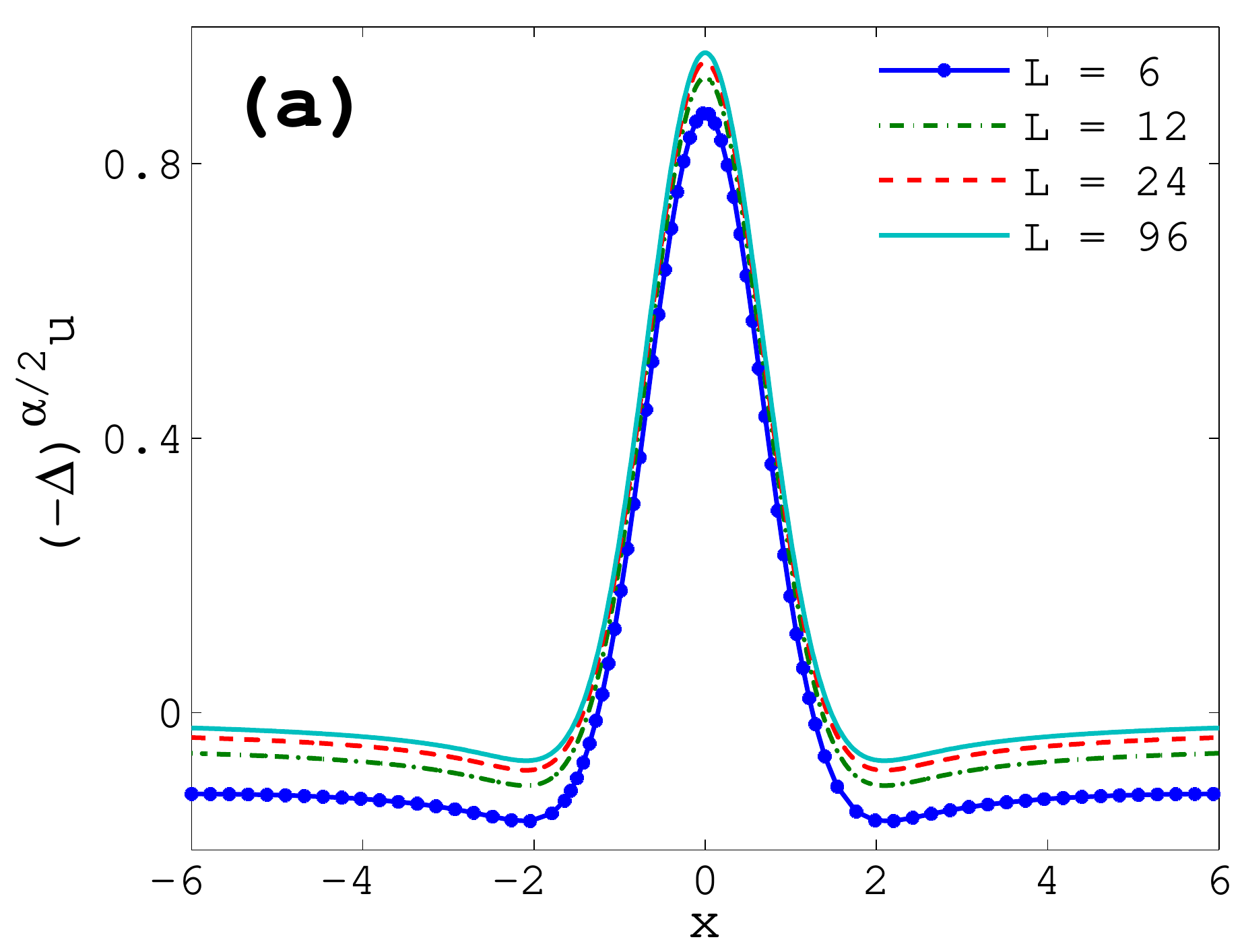}$~$
  \includegraphics[totalheight=0.27\textheight]{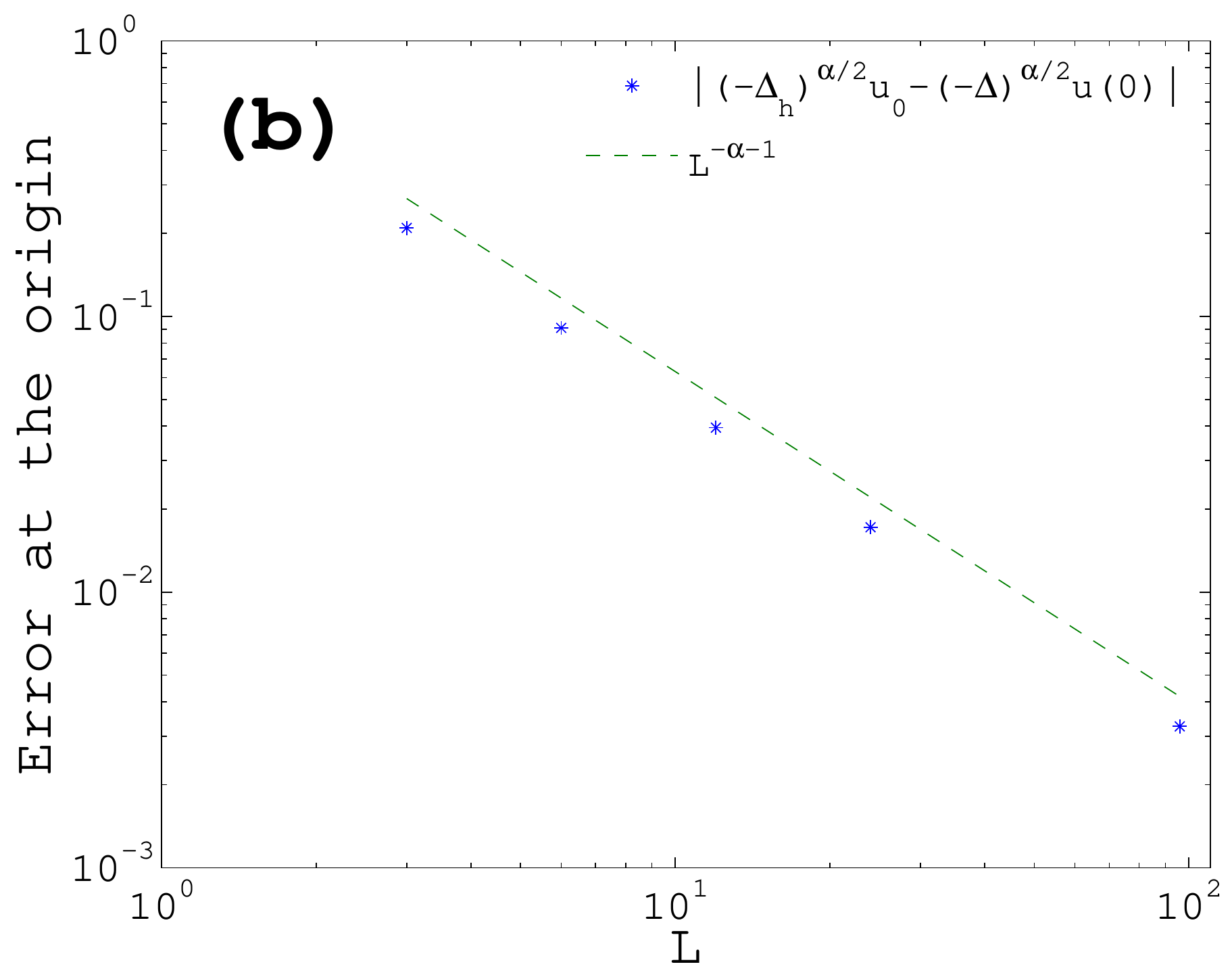}
 \end{center}
\caption{The slow convergence of the fractional Laplacian based on FFT: (a)
the fractional Laplacian of $u(x)=e^{-x^2}$ for $\al=0.2$ computed on different
domain sizes (but only shown on $[-6,6]$); (b) the error
at the origin that decays like $O(L^{-\alpha-1})$.}
\label{fig:fft}
\end{figure}

While the conventional Fourier spectral method is not satisfactory as expected,
the correct one in the form of~\eqref{FLh} can be constructed
by choosing the exact symbol $M(\xi)=|\xi|^\alpha$.
In this way, the weights defined via~\eqref{eq:SpectralWeight} become
\begin{equation}\label{eq:spweight}
w_k^{SP} = -\frac{h^{-\alpha}}{\pi}\int_0^\pi\! \xi^\alpha \cos (k\xi)\, \mathrm{d}\xi.
\end{equation}
Exactly the same weights $w_k^{SP}$ can be derived alternatively
from the well-known $\sinc$ interpolation $\mathcal{P}u$ of $u$,
when $(-\Delta)^{\alpha/2}u(x_j)$ is approximated by
$(-\Delta)^{\alpha/2}\mathcal{P}u(x_j)$. Here the $\sinc$ interpolation $\mathcal{P}u$ is
defined by
\[
\mathcal{P}u(x) = \sum_{k=-\infty}^\infty u_k \sinc \frac{x-x_k}{h},
\]
with the $\sinc$ function  $\sinc x = \frac{\sin \pi x}{\pi x}$ that has been explored
extensively for the numerical solutions of differential equations~\cite{MR1171217,MR1226236}.
The function $\mathcal{P}u$ is an interpolation of $u$ on $\mathbb{R}$ in the sense that
$\mathcal{P}u(x_j)=u_j$ for any $j \in \mathbb{Z}$. Once $\mathcal{P}u$ is defined,
the same weights~\eqref{eq:spweight} can be derived,
using the crucial property that  the Fourier transform of $\sinc x$ is
precisely $\chi_{[-\pi,\pi]}(\xi)$,  the characteristic function on $[-\pi,\pi]$.
In this way, the Fourier transform of $\mathcal{P}u$ is
\[
\mathcal{F}[\mathcal{P}u](\xi) = h \chi_{[-\pi/h,\pi/h]}(\xi)\sum_{k=-\infty}^\infty u_k e^{-i\xi x_k}
\]
and the discrete operator defined by $(-\Delta)^{\alpha/2} [\mathcal{P}u](x_j) $
can be written as
\begin{equation*} 
\frac{1}{2\pi}\int_{-\infty}^\infty\! e^{i\xi x_j}|\xi|^\alpha \mathcal{F}[\mathcal{P}u](\xi)\,
\mathrm{d}\xi = \frac{h}{2\pi}\int_{-\pi/h}^{\pi/h}\!
\Big(\sum_{k=-\infty}^\infty u_k e^{i\xi(x_j-x_k)}\Big) |\xi|^\alpha\,  \mathrm{d}\xi.
\end{equation*}
Finally using the identity~\eqref{eq:invsemiFFT}, we get
\[
 \frac{h}{2\pi} \int_{-\pi/h}^{\pi/h} \left(
\sum_{k=-\infty}^{\infty} u_j e^{i\xi (x_j-x_k)}
\right)|\xi|^{\alpha}\,\mathrm{d}\xi
=u_j e^{i\xi x_j}\int_{-\pi/h}^{\pi/h}
\left(\frac{h}{2\pi}
\sum_{k=-\infty}^{\infty}  e^{-i\xi x_k}
\right) |\xi|^{\alpha}\,\mathrm{d}\xi=0.
\]
Therefore, by combining the previous two equations,
$(-\Delta)^{\alpha/2} [\mathcal{P}u](x_j)$  can be simplified as
\begin{align*}
 &\quad \frac{h}{2\pi}\int_{-\pi/h}^{\pi/h}\!\Big(\sum_{k=-\infty}^\infty u_k e^{i\xi(x_j-x_k)}\Big) |\xi|^\alpha\,  \mathrm{d}\xi
-\frac{h}{2\pi}\int_{-\pi/h}^{\pi/h}\!\Big(\sum_{k=-\infty}^\infty u_j
 e^{i\xi(x_j-x_k)}\Big) |\xi|^\alpha\,  \mathrm{d}\xi
\cr
&= \sum_{k=-\infty}^\infty (u_k-u_j)\frac{h}{2\pi}\int_{-\pi/h}^{\pi/h}\! e^{i\xi(x_k-x_j)} |\xi|^\alpha\, \mathrm{d}\xi \cr
&=-\sum_{k=-\infty}^\infty (u_j-u_{j-k})\frac{h^{-\alpha}}{2\pi}\int_{-\pi}^{\pi}\! |\xi|^\alpha\cos k\xi\, \mathrm{d}\xi.
\end{align*}
Comparing the last expression with~\eqref{FLh}, we recover the same weights $w^{SP}_k$ given
by~\eqref{eq:spweight}. Since the symbol $M^{SP}(\xi)$ is exactly $|\xi|^{\alpha}$, we call $w^{SP}$ the spectral weights.

\begin{figure}[htp]
 \begin{center}
  \includegraphics[totalheight=0.27\textheight]{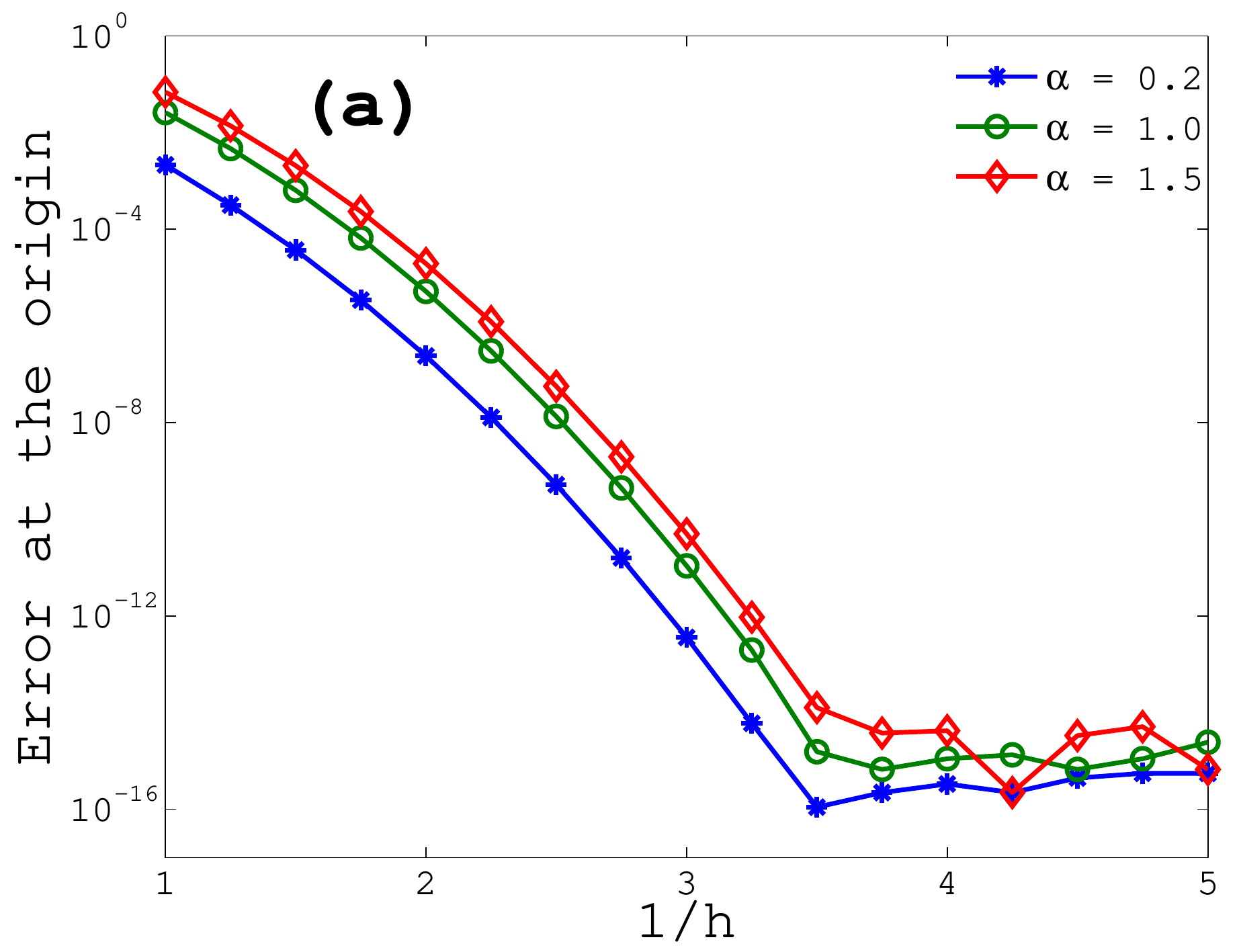}
  $~$\includegraphics[totalheight=0.27\textheight]{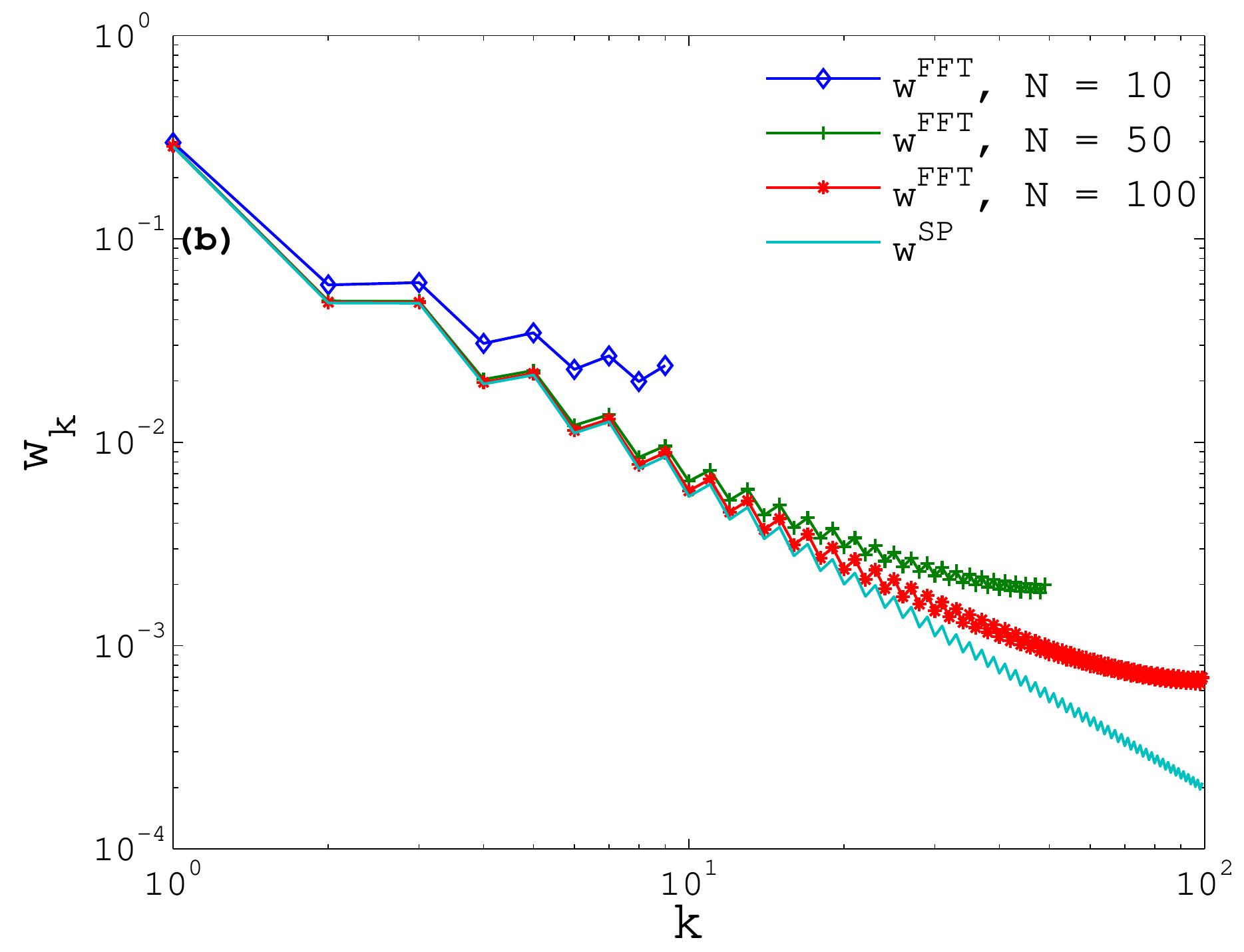}
 \end{center}
\caption{(a) The spectral accuracy of the spectral weights $w^{SP}$ for the fractional Laplacian of $e^{-x^2}$
at the origin $\alpha =0.2$ with $\alpha = 1.0$ and $\alpha=1.5$. (b) The spectral weights $w^{SP}$ and
the weights extracted from FFT (normalized by the grid size $h^{\alpha}$) for $\alpha=0.5$.
}
\label{fig:spectralconv}
\end{figure}

The spectral convergence rate of the weights $w^{SP}$ is verified in Figure~\ref{fig:spectralconv}(a),
for the fractional Laplacian of $u(x)=e^{-x^2}$ at the origin with different $\alpha$.
The computational domain is fixed to be $[-8,8]$ and the accuracy does not improve on
larger domains for fixed grid size $h$.  By a careful analysis, FFT-based scheme for the
fractional Laplacian also takes the form~\eqref{FLh}, and the  weights\footnote{These weights can be taken as the output
of the FFT-based method on the discrete Dirac delta function.}
are trapezoidal rule approximation of the integral~\eqref{eq:spweight}.
Consequently, the weights from FFT-based scheme converge to the spectral weights
when the number of points used in FFT goes to infinity, as shown in~\ref{fig:spectralconv}(b).

\begin{rem}
When $k$ is large, it is difficult to evaluate highly oscillatory integrals like~\eqref{eq:spweight}
accurately, and alternative expressions using special functions are preferred.
In terms of the lower incomplete Gamma function $\gamma(a,z) = \int^z_0 t^{a-1}e^{-t}dt$,
the anti-derivative of $\xi^{\alpha}e^{ik\xi}$ can be evaluated as
$(-ik)^{-\alpha-1}\gamma(\alpha+1,-ik\xi)$.
Hence the integral~\eqref{eq:spweight} can be written as
\begin{align}
  \label{eq:lommelsp}
  w^{SP}_k
&  = \frac{h^{-\alpha}}{\pi} \mbox{Re}\big(
  (-ik)^{-\alpha-1}\big[\Gamma(\alpha+1,-\pi ik)-\Gamma(\alpha+1,0)\big]
  \big) \cr
  & =    -\frac{h^{-\alpha}\pi^{\alpha}}{\alpha+1} \mbox{Re}\big({
  }_1F_1(\alpha+1;\alpha+2;ik\pi)\big),
  \end{align}
in terms of confluent hypergeometric function ${ }_1F_1$. These special functions
involved can be evaluated much more efficiently with existing
packages~\cite{MR2383071,MR1406797}.
\end{rem}

\begin{rem}
The weights $w_k^{SP}$ for $\alpha=1$ or $2$ can be obtained in elementary means,
that is,
\begin{equation}
    \label{eq:SpecialSpectralWeight}
w^{SP}_k =
\begin{cases}
\big(1-(-1)^k\big)\big/{k^2\pi h}
, & \alpha = 1,
 \\[0.1em]
 2(-1)^{k+1}\big/{k^2h^2} ,
& \alpha = 2,
\end{cases}
\end{equation}
where the case $\alpha=2$ already appeared
in~\cite[equation (2.14)]{MR1776072} for
spectrally accurate discretization of the second derivative.
These special cases indicate that the desired non-negativity of the
weights may be lost; in fact, it can be shown numerically that the weights
have alternating signs for $\alpha$ between $1$ and $2$.

\end{rem}

\subsection{Weights from regularized Fourier symbol}
The above scheme with the spectral weights $\{w_k^{SP}\}$,
despite its high accuracy for smooth functions, has some
undesired properties. Because of the presence of negative
weights for $\alpha>1$, the discrete fractional Laplacian
of functions defined on $\mathbb{Z}_h$ is not necessarily negative
at its global maximum. Negative densities could  arise in the
explicit Euler scheme~\eqref{eq:disfracheat} for the fractional heat
equation. Moreover, in the limit as $\alpha$ approaches $2^-$,
the scheme with limiting weights shown in~\eqref{eq:SpecialSpectralWeight}
does not agree with any common finite difference approximation of $-\partial_{xx}$
on a compact stencil.

On the other hand, to reinforce the limiting behaviour as $\alpha \to 2^{-}$,
the general scheme  can be constructed as follow:
take the symbol at $\alpha=2$ and define the new symbol for
$\alpha \in (0,2)$  by raising to power $\alpha/2$.
For instance, the rescaled symbol to the standard three-point
central difference method of $-\partial_{xx}$ is $2(1-\cos(\xi))$
and thus the symbol for general $\alpha$ between $0$ and $2$ is just
\begin{equation}\label{eq:regsymb}
    \qquad M^{PER}(\xi) = \left[2\big(1-\cos(\xi)\big)\right]^{\alpha/2},
\end{equation}
with the weights
\begin{equation*}
w^{{PER}}_k = -\frac{h^{-\alpha} }{\pi}  \int_0^\pi
\!  \left[2\big(1-\cos(\xi)\big)\right]^{\al/2} \cos (k\xi) \,\mathrm{d} \xi.
\end{equation*}
In fact, these  weights have already been calculated explicitly
in~\cite[equation (7)]{FLBD},
\begin{equation}\label{eq:expPERweights}
w^{{PER}}_k = h^{-\alpha}\frac{\Gamma\big(k-\frac{\alpha}{2}\big)
\Gamma(1+\alpha)}{\pi \Gamma\big(k+1+\frac{\alpha}{2}\big)}
\sin\big(\frac{\alpha \pi}{2}\big), \quad k>0.
\end{equation}
Since $\Gamma(z)>0$ for any $z>0$, this closed form expression clearly shows
the non-negativity of the weights for $\alpha \in (0,2)$.


Following the same spirit, other schemes can be constructed
based on higher order difference approximation of $-\partial_{xx}$.
For example, from the five-point central difference
\[
-u''(x_j)\approx    \frac{1}{h^2}\left(
    \frac{1}{12}u(x_j+2h)-\frac{4}{3}u(x_j+h)+\frac{5}{2}u(x_j)
    -\frac{4}{3}u(x_j-h)+\frac{1}{12}u(x_j-2h)
    \right) ,
\]
we can choose $M(\xi)=\left(\frac{5}{2}-\frac{8}{3}\cos\xi+\frac{1}{6}\cos
2\xi\right)^{\alpha/2}$.
However, this approach of pursing higher order schemes is limited
for at least two reasons. First, the resulting weights may become negative
when $\alpha$ is close to $2$.
Second, the integral~\eqref{eq:SpectralWeight}
defining the weights from the symbol is highly oscillatory for large $k$,
and the error in the evaluation of the weights can contaminate
the accuracy of the underlying scheme, unless closed form
expressions like~\eqref{eq:spweight} or~\eqref{eq:expPERweights}
are available.

\begin{rem} The superscript $PER$ is chosen because the rescale symbol
    $M^{PER}(\xi)$  can be extended smoothly as a periodic function near
    the boundary $\pm \pi$ (though still not smooth near the origin).
\end{rem}


\subsection{Gorenflo-Mainardi scheme using \GL weights}

In one dimension, the fractional Laplacian can be expressed as
the symmetric combination of two fractional derivatives as in~\eqref{eq:fracdevlap},
and hence numerical schemes for fractional Laplacians can be constructed from
the classical Gr\"unwalkd-Letnikov differences
originated from fractional derivatives.  The corresponding weights, proposed by
Gorenflo and Mainardi~\cite{gorenflo1998random,gorenflo1999discrete,gorenflo2002discrete}, are given by
\begin{subequations}
\begin{equation}\label{eq:GLa01}
w^{GL}_k = \frac{h^{-\al} }{2\cos (\alpha \pi /2)}
\frac{\alpha \Gamma(k-\alpha)}{k!\ \Gamma(1-\alpha)} ,\qquad k \geq 1,
\end{equation}
for $\alpha \in (0,1)$ and
\begin{equation}\label{eq:GLa12}
w^{GL}_k =
\begin{cases}
-\frac{h^{-\al}}{2\cos (\alpha\pi/2)}\Big[1+\alpha(\alpha-1)/2\big],
& k=1, \\[0.2em]
\frac{h^{-\al}}{2\cos (\alpha\pi/2)}\frac{\alpha
\Gamma(k+1-\alpha)}{(k+1)!\ \Gamma(1-\alpha)}, \quad\qquad &  k>1,
\end{cases}
\end{equation}
\end{subequations}
for $\alpha \in (1,2]$.  Notice that the weights in the range $\alpha \in
(1,2)$ are just a simple shift of those for $\alpha \in (0,1)$, to preserve
the non-negativity of the weights and therefore many other desired properties.
The case $\al = 1$ is handled separately, because it can not be continued
from the above two cases~\eqref{eq:GLa01} and~\eqref{eq:GLa12} by taking
the limits $\alpha \to 1^-$ and $\alpha \to 1^+$ respectively;
for this special case, one option is to choose (see~\cite{gorenflo2002discrete}
for more details)
\[
w^{GL}_k = \frac{1}{\pi h}  \frac{1}{k (k+1)}, \qquad k>0.
\]


The rescaled symbol $M^{GL}(\xi)$ associated with the weights~\eqref{eq:GLa01}
and~\eqref{eq:GLa12} can be computed explicitly,
based on the  binomial expansion of $(1-z)^{\alpha}$. That is,
\begin{subequations}
\begin{align}
 (1-z)^{\alpha} &= 1-\frac{\alpha}{\Gamma(1-\alpha)}\sum_{k=1}^\infty
 \frac{\Gamma(k-\alpha)}{k!}z^k \label{eq:bexp1} \\
 &= 1-\alpha z + \frac{1}{2}\alpha(\alpha-1)z^2-\frac{\alpha}{\Gamma(1-\alpha)}
\sum_{k=2}^\infty \frac{\Gamma(k+1-\alpha)}{(k+1)!} z^{k+1}.\label{eq:bexp2}
\end{align}
\end{subequations}
By substituting $z=1$ into~\eqref{eq:bexp1} and~\eqref{eq:bexp2}, we obtain
\begin{equation}\label{eq:GLw0}
w_0^{GL} = -\sum_{k\neq 0} w^{GL}_k
= \begin{cases}
  -[h^{\alpha}\cos (\alpha\pi/2)]^{-1}, \qquad & \alpha \in (0,1),\cr
  \alpha [h^{\alpha}\cos (\alpha\pi/2)]^{-1}, & \alpha \in (1,2].
\end{cases}
\end{equation}
Once $w_0^{GL}$ is determined, the rescaled symbol $M(\xi)$ is obtained
by rearranging the terms in~\eqref{eq:bexp1} and~\eqref{eq:bexp2}
with $z=e^{\pm i\xi}$, that is,
\begin{align*}
    M^{GL}(\xi) = -h^{\alpha}\sum_{k=-\infty}^\infty e^{-ik\xi}w_k^{GL}
 &=\begin{cases}
   \frac{1}{2\cos \alpha \pi/2}
   \left[(1-e^{i\xi})^\alpha+(1-e^{-i\xi})^\alpha\right],
   \qquad &\alpha \in (0,1),\cr
   \frac{1}{2\cos \alpha\pi/2}
   \left[e^{-i\xi}(1-e^{i\xi})^\alpha
   +e^{i\xi}(1-e^{-i\xi})^\alpha\right],&\alpha \in (1,2].
 \end{cases}
\end{align*}
Here the branch of the multivalued functions $(1-e^{\pm i\xi})^\alpha$
are chosen so that it is  consistent with the expansions~\eqref{eq:bexp1}
and~\eqref{eq:bexp2}. Using trigonometric identities,
$M^{GL}(\xi)$ can be further simplified as
\begin{equation}\label{eq:GLM}
    M^{GL}(\xi) = \begin{cases}
\frac{\cos\frac{\pi-|\xi|}{2}\alpha}{\cos\frac{\alpha\pi}{2}}
\left(2\sin\frac{|\xi|}{2}\right)^\alpha, \qquad &
\alpha \in (0,1),\cr
\frac{\cos\left(\frac{\pi-|\xi|}{2}\alpha+|\xi|\right)}{\cos\frac{\alpha\pi}{2}}
\left(2\sin\frac{|\xi|}{2}\right)^\alpha, \qquad &
\alpha \in (1,2].
\end{cases}
\end{equation}

However, this approach based on fractional derivatives is
restricted to be in one dimension. The straightforward extension
to $\mathbb{R}^n$ by taking summations of fractional Laplacian
in each dimension gives rise to the symbol $|\xi_1|^{\alpha}
+ |\xi_2|^{\alpha} + \cdots + |\xi_n|^{\alpha}$,
different from the one $|\xi|^{\alpha}=\big(|\xi_1|^2+\cdots+|\xi_n|^2\big)^{\alpha/2}$
of our interests here.

\subsection{The finite difference-quadrature method}
Finally we discuss the weights derived from quadratures based on the
singular integral representation~\eqref{eq:rieszfrac}.
This approach has been taken in the context of finite volume or
finite difference methods~\cite{MR2552219,MR2795714,MR2832791},
where $u$ is assumed to be the
constant $u(x_j)$ on the interval $(x_j-h/2,x_j+h/2)$. In this way,
\begin{align*}
    (-\Delta)^{\alpha/2}u(x_j) &=
    C_{1,\alpha}\sum_{k=-\infty}^\infty \int_{x_{j-k}-h/2}^{x_{j-k}+h/2}
    |x_j-y|^{-1-\alpha}\big(u(x_j)-u(y)\big)    \dr y \cr
    &\approx \sum_{k\neq 0} (u_j-u_{j-k})C_{1,\alpha}
    \int_{x_{j-k}-h/2}^{x_{j-k}+h/2} |x_j-y|^{-1-\alpha} \dr y,
\end{align*}
from which the weights $w_k$ can be extracted. However,
the singularity when $y\approx x_j$ is in general not treated
appropriately, leading to
inconsistency as $\alpha$ goes to $2^-$ (a recent correction is made
in~\cite{droniou2014uniformly}).

An improved  difference-quadrature scheme is proposed by the authors
in~\cite{HuangOberman1},
using piecewise local polynomial basis. Near the singularity $y\approx x_j$ in
the integral~\eqref{eq:rieszfrac},
$u(y)$ is replaced by the (second order) Taylor expansion of $u(x)$ at $x_j$,
while away from the singularity, $u(y)$ is approximated by piecewise linear or quadratic polynomials.
Here we only quote the expressions of the weights and refer to~\cite{HuangOberman1} for more details.
Define the auxiliary functions $F(t)$ and $G(t)$ such that
$F''(t) = G'''(t)=t^{-1-\alpha}$, or equivalently,
\[
F(t) =
\begin{cases}
\frac{1}{(\al-1)\al} \abs{t}^{1-\al},& \al \not = 1
\\
- \log |t|, & \al = 1
\end{cases},\qquad\quad
G(t) =
\begin{cases}
\frac{1}{(2-\al)(\al-1)\al} \abs{t}^{2-\al},& \al \not = 1
\\
t-t\log |t|,  & \al = 1
\end{cases}.
\]
Then the weights using piecewise linear functions (Tent functions) are given by
\[
 w_k^T=
\Cal h^{-\alpha}
 \begin{cases}
\dfrac{1}{2-\alpha}
-F'(1)+F(2)-F(1), &k = 1, \\
F(k+1)-2F(k)+F(k-1), & k = 2, 3, \cdots,
\end{cases}
\]
and the weights using piecewise quadratic functions are given by
\[
    \frac{h^{\alpha}w_k^Q}{\Cal} =
 \begin{cases}
\dfrac{1}{2-\alpha}-G''(1)-\dfrac{G'(3)+3G'(1)}{2}+G(3)-G(1), & k = 1,
\\
 2\big[  G'(k+1)+G'(k-1)
 - G(k+1)+G(k-1)\big], & k = 2, 4, 6, \cdots,
\\
 -\dfrac{G'(k+2)+6G'(k)+G'(k-2)}{2}
 +G(k+2)-G(k-2), & k = 3,5,7,\cdots.
\end{cases}
\]
Piecewise cubic or higher order polynomials can also be used,
but the resulting weights may also become negative (not surprisingly again
when $\alpha$ is close to 2) and hence further study in this direction
is not pursued. Different from the previous three cases,
the rescaled symbol $M(\xi)$ for either $w_k^T$ or $w_k^Q$
does not seem to possess a closed form expression.

With different weights and their symbols in mind, we are now in a position
to compare them so that we can chose the right one in practice.

\section{Discussions and Comparisons of the weights}\label{sec:DCweights}
In this section we compare and discuss more specific properties of the
weights presented in the last section,
especially their order of accuracy,  ultimately important for different choices of the schemes.

\subsection{Basic properties of  the weights}
Properties shared by most of the schemes are discussed in this subsection, including positivity,
asymptotic decay rates as the indices go to infinity, the limiting behaviours as $\alpha$ goes to $2$
 and CFL conditions in the discretisation of evolution equations.
The orders of accuracy of these schemes are discussed in the next subsection,
using a local expansion of the rescaled symbol $M(\xi)$.

\subsubsection*{Positivity and decay rate}
\label{sec:weightprop}
The non-negativity of the weights (except $w_0$)
is essential in practical discretisation of elliptic or parabolic PDEs to preserve
the discrete maximum principle or comparison principle.
The weights $w^{PER}$ and $w^{GL}$ are clearly positivity,
from the explicit formulas of the weights~\eqref{eq:expPERweights},
~\eqref{eq:GLa01} and~\eqref{eq:GLa12}.
The weights $w^T$ and $w^Q$ are also  shown to be strictly positive
in~\cite{HuangOberman1}, but not necessarily true when higher order polynomial
interpolations are used. In contrast, the weights
$w^{SP}$ are strictly positive only for $\alpha \in (0,1)$;
when $\alpha \in (1,2)$, $w_k^{SP}$ is negative for $k$ even, making it
less attractive in many applications despite its spectral accuracy.

The asymptotic decay rates of these weights $w_k$ as the index $k$ goes
to infinity can also be obtained. From asymptotic expansion
$\Gamma(z)\sim\sqrt{2\pi}z^{z-1/2}e^{-z}$ for Gamma function $\Gamma(z)$ with
large argument $z$~\cite{MR1225604}, it is easy to see that
$\Gamma(z+a)/\Gamma(z+b) \sim z^{a-b}$ and all three ratios
\[
 \frac{\Gamma(k-\frac{\alpha}{2})}{\Gamma(k+1+\frac{\alpha}{2})},\quad
 \frac{\Gamma(k-\alpha)}{k!},\quad
 \frac{\Gamma(k+1-\alpha)}{(k+1)!}
\]
have the same leading asymptotics $k^{-\alpha-1}$.
Consequently,
\[
  w_k^{PER} \sim \frac{\Gamma(1+\alpha)}{\pi}\sin\left(\frac{\alpha\pi}{2}\right)
 k^{-\alpha-1}h^{-\alpha}\quad \mbox{ and }\quad w_k^{GL} \sim
 \frac{\alpha}{2\cos \frac{\alpha\pi}{2}}\frac{1}{\Gamma(1-\alpha)}k^{-\alpha-1}h^{-\alpha}.
\]
Similarly, from the expressions for $w^T$ and $w^Q$ given in the previous section,
\[
    w_k^T \sim C_{1,\alpha} k^{-\alpha-1}h^{-\alpha},
    \qquad
    w_k^Q \sim \begin{cases}
        \frac{4}{3}C_{1,\alpha} k^{-\alpha-1}h^{-\alpha},\quad
        &k \mbox{ odd},\\
        \frac{2}{3}C_{1,\alpha} k^{-\alpha-1}h^{-\alpha}
        & k \mbox{ even}.
    \end{cases}
\]
In fact, the constant prefactors to $k^{-\alpha-1}h^{-\alpha}$ are exactly
the same for all the weights $w^{PER}$, $w^{GL}$ and $w^{T}$. Using the well-known
Euler's reflection formula $\Gamma(z)\Gamma(1-z)=\pi/\sin z$
and the duplication formula $\Gamma(z)\Gamma\big(z+\frac{1}{2}\big)
=2^{1-2z}\sqrt{\pi}\Gamma(2z)$~\cite{MR1225604}, we have
\[
  \frac{\Gamma(1+\alpha)}{\pi} \sin\left(\frac{\alpha\pi}{2}\right)
  =\frac{\alpha}{2\cos \frac{\alpha\pi}{2}}\frac{1}{\Gamma(1-\alpha)}
  =\frac{\alpha 2^{\alpha-1}\Gamma(\frac{1+\alpha}{2})}
  {\pi^{1/2}\Gamma(\frac{2-\alpha}{2})} (=C_{1,\alpha}).
\]

The asymptotic decay rates are more complicated for $w^{SP}$.
We can show numerically that, for $\alpha \in (0,1)$, the weights $w_k^{SP}$ have the same scaling 
$k^{-\alpha-1}h^{-\alpha}$ as others, while for $\alpha \in (1,2)$,
the weights $w_k^{SP}$ scale like $k^{-2}h^{-\alpha}$ with
alternating signs.
The scaling behaviour of the weights is verified in Figure~\ref{fig:weights}.

\begin{figure}[htp]
 \begin{center}
  \includegraphics[totalheight=0.27\textheight]{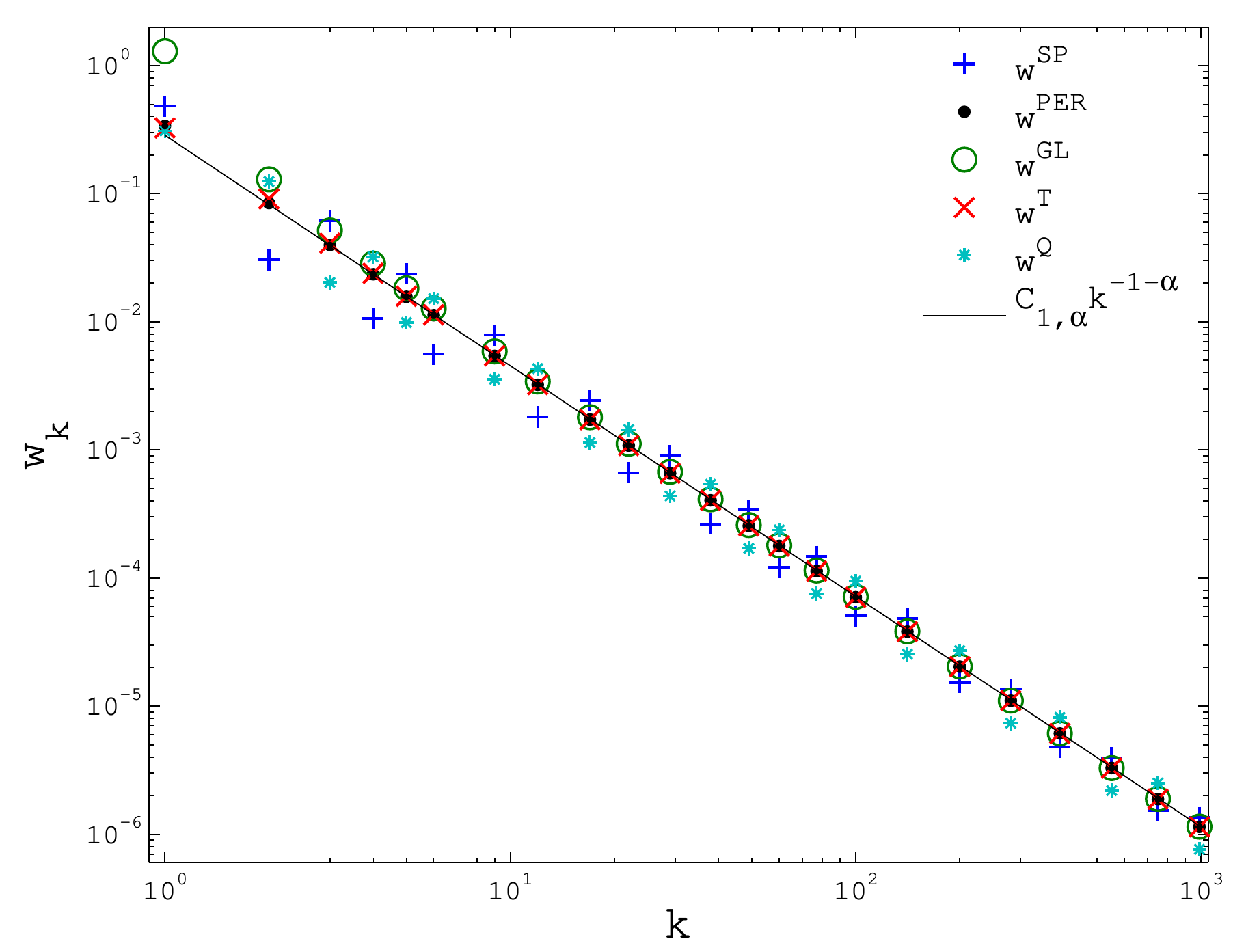}$~$
  \includegraphics[totalheight=0.27\textheight]{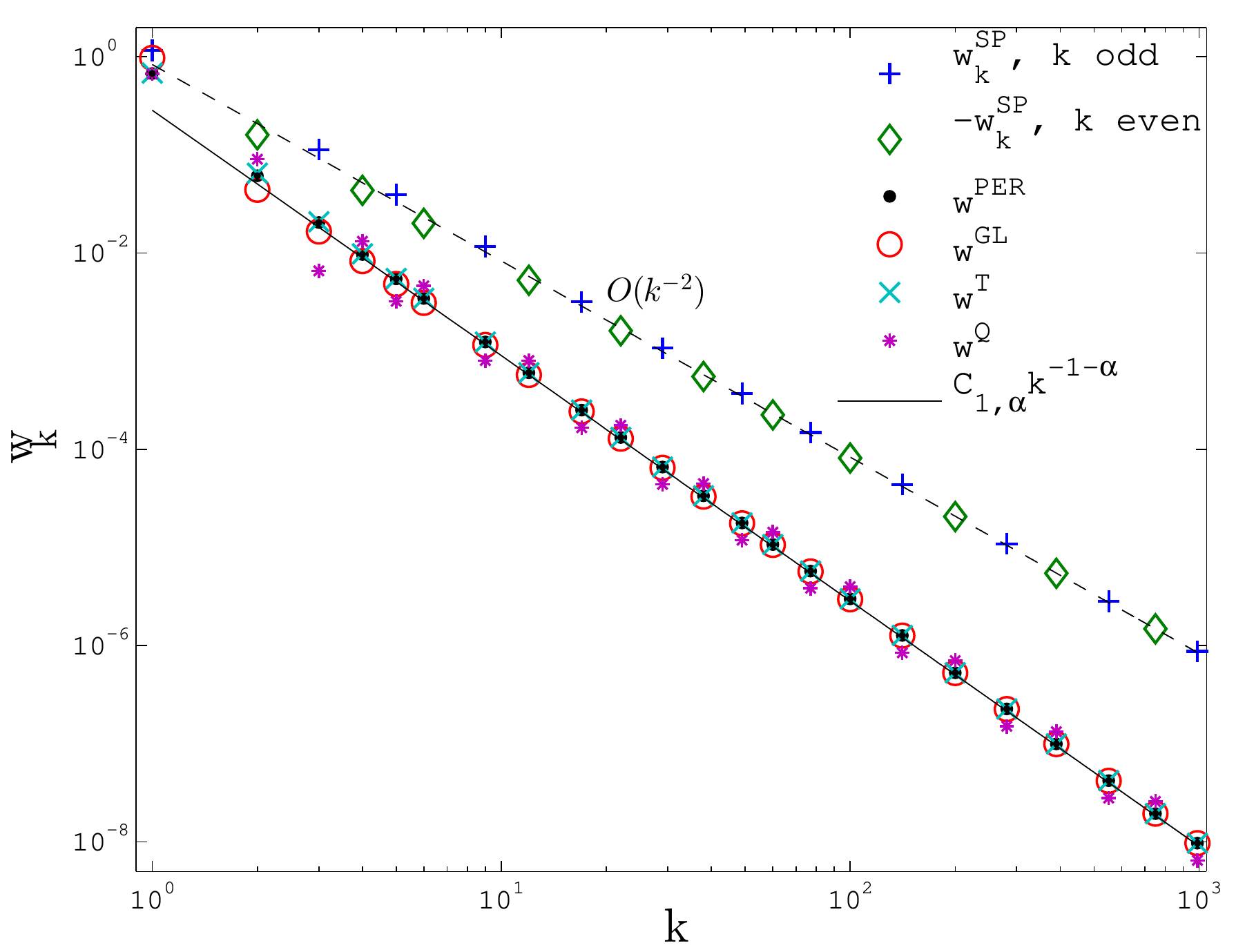}
 \end{center}
\caption{The decay of the weights $w_k$ of different methods
    (without the factor $h^{-\alpha}$), for {\bf (a)} $\alpha = 0.8$
    and {\bf (b)} $\alpha = 1.5$. All positive weights decay with the same rate
    $k^{-1-\al}$, while the rate for $w^{SP}$ as $\alpha\in(1,2)$ is
    $O(k^{-2})$.}
\label{fig:weights}
\end{figure}

\begin{rem} The scaling behaviour of the Fourier integral $\int_{-\pi}^{\pi}
  M(\xi)e^{ik\xi}\dr \xi$  in $k$
  is intimately related to the singularity of $M(\xi)$ on the interval $[-\pi,\pi]$,
analogous to the relation between the smoothness of a function
and the decay of its Fourier transform (or Fourier series).
Using asymptotic analysis on~\eqref{eq:spweight}, one can show that
$w_{k}^{SP}$ decays like $k^{-2}$ for $\alpha \in (1,2)$, instead of
the expected rate $k^{-\alpha-1}$ shared by other weights.
\end{rem}
\subsubsection*{The limit when $\alpha \to 2^-$}
As $\alpha$ goes to $2^-$, the fractional Laplacian $(-\Delta)^{\alpha/2}$
becomes the standard (negative) Laplacian.  Hence the corresponding
numerical scheme~\eqref{FLh} is expected to converge to some approximation
of $-\partial_{xx}$, preferably the three-point central difference
$(2u_{j}-u_{j+1}-u_{j-1})/h^2$. The limiting weights can be calculated directly from their
explicit expressions presented in the last section.
From~\eqref{eq:SpecialSpectralWeight}, $w^{SP}$ becomes the
spectral method on the whole real line~\cite{MR1776072}.
All other weights $w^{PER}$, $w^{GL}$, $w^T$ and $w^Q$ converge
to the limiting scheme with $w_{\pm}=1$ and $w_k=0$ for $k>1$,
 mainly using the fact that $\Gamma(z)\approx 1/z$ as $z\to 0$.
Therefore, the standard three-point central difference in the limit
$\alpha \to 2^-$ is recovered for all weights except $w^{SP}$.


\subsubsection*{CFL condition} When an evolution equation is discretised,
the time step is usually restricted by the so called \emph{CFL condition}
for stability reason. Taking the fractional heat equation $u_t+(-
\Delta)^{\alpha/2}u=0$ for example,
the forward Euler scheme at time level $n+1$ is
\[
    u_{j}^{n+1}=(1+w_0\Delta t )u_j^n + \Delta t
    \sum_{k\neq 0} w_ku_{j-k}^n.
\]
Beside that the weights should be non-negative, the time
step is also restricted by the CFL condition $1+w_0\Delta t \geq 0$, i.e.,
\[
    \frac{\Delta t}{h^{\alpha}} \leq
    C_{\mbox{max}}:=-\frac{1}{h^{\alpha} w_0}
     =\frac{1}{h^{\alpha}\sum_{k\neq 0}w_k}.
\]
As a result, we examine the special weight $w_0$ for different schemes below.

If the rescaled symbol $M(\xi)$ satisfies the condition $M(0)=0$,
$w_0$ is given directly from the integral~\eqref{eq:SpectralWeight}
(by taking $k=0$). This immediately implies that
\[
    w_0^{SP} = -\frac{h^{-\alpha}}{\pi}\int_0^\pi\!
    \xi^\alpha\,\mathrm{d}\xi =
    -\frac{\pi^{\alpha}}{1+\alpha}h^{-\alpha},
\]
and
\[
    w_0^{PER} = -\frac{h^{-\alpha}}{\pi}\int_0^\pi \!
    \Big[2(1-\cos\xi)\Big]^{\alpha/2}\,\mathrm{d}\xi
    =-\frac{4\Gamma(\alpha)}{\alpha\Gamma(\alpha/2)^2}h^{-\alpha},
\]
together with $w_0^{GL}$ already derived in~\eqref{eq:GLw0}.
Though the rescaled symbol $M(\xi)$ is not explicit,
the weights $w^T$ or $w^Q$ constitute a telescoping series,
and
\[
w_0^T= w_0^Q=
-2C_{1,\alpha}h^{-\alpha}\left[\frac{1}{2-\alpha}-F'(1)\right]
=-\frac{2^{\alpha}\Gamma(\frac{1+\alpha}{2})}{\pi^{1/2}\Gamma(2-\frac{\alpha}{2})}h^{-\alpha}.
\]


The constant $C_{\max}=\big(-h^{\alpha}w_0\big)^{-1}
=\big(h^{\alpha}\sum_{k\neq 0}w_k\big)^{-1}$,
as a function of $\alpha$ is shown in Figure~\ref{fig:sumW} for different
schemes. It is obvious that the CFL condition
becomes too restrictive only for $w^{GL}$ near $\alpha =1$, where
the time step has to be vanishingly small. Because of the alternating
signs of the weights $w^{SP}_k$ for $\alpha \in (1,2)$, the CFL condition
is only shown on the interval $\alpha \in (0,1)$.

\begin{figure}
 \begin{center}
  \includegraphics[totalheight=0.3\textheight]{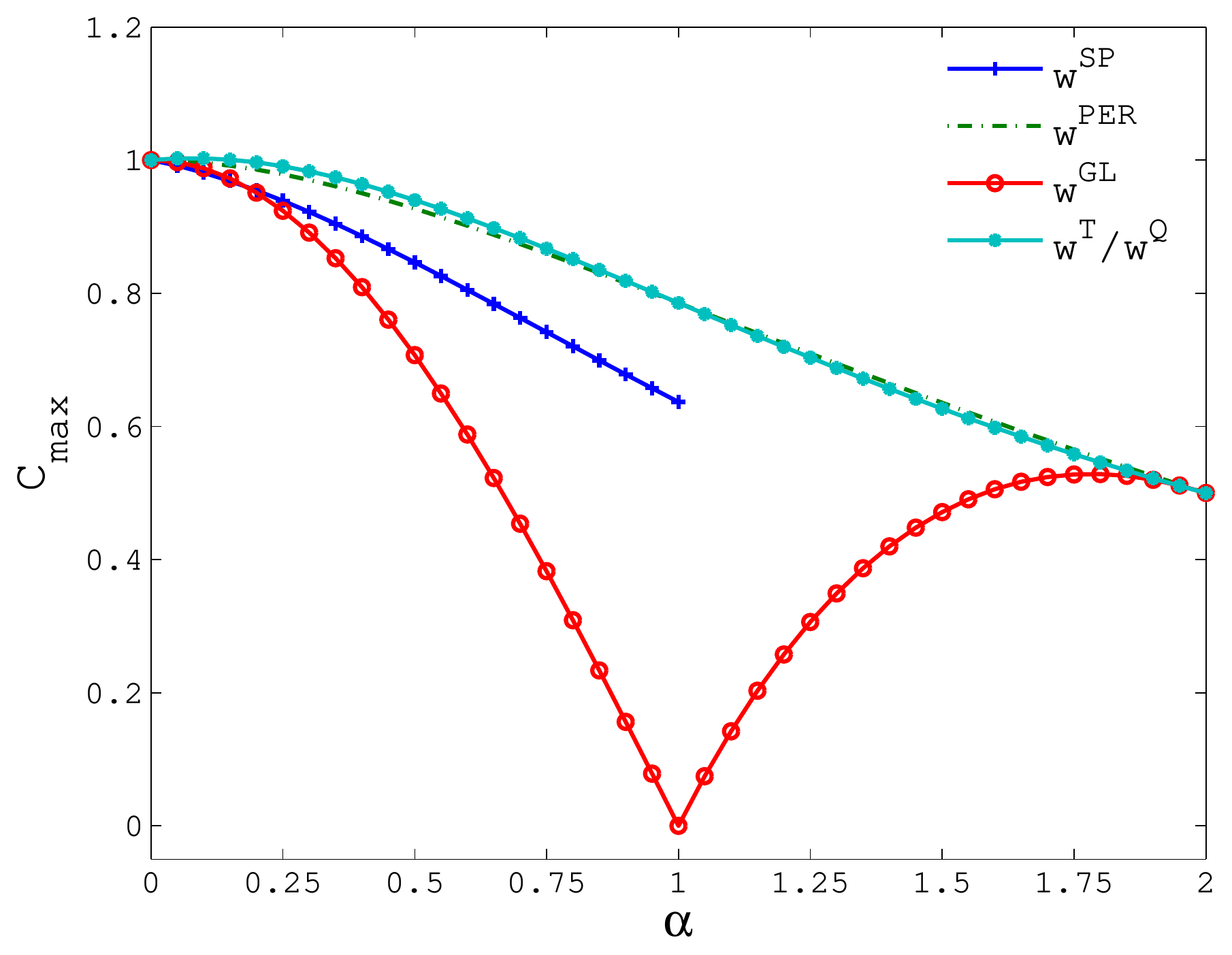}
 \end{center}
 \caption{The CFL condition $C_{\mbox{max}}=\big(h^\alpha\sum_{k\neq
 0}w_k)^{-1}=\big(-h^{\alpha}w_0\big)^{-1}$ for different schemes.
 The spectral weights $w^{SP}$ is only plotted for $\alpha \in (0,1)$, as the weights
 become negative for $\alpha (1,2)$.
 }
\label{fig:sumW}
\end{figure}

\subsection{Order of accuracy via the rescaled symbols}
\label{sec:weightsFourier} In general, it is difficult to assess the accuracy of
finite different schemes for nonlocal operators, precisely because
the nonlocality prevents conventional approaches using Taylor expansions.
Nevertheless, under modest conditions on the decay of the underlying functions in Fourier space,
the order of accuracy of the scheme~\eqref{FLh} can be computed formally,
that is,
\[
    (-\Delta)^{\alpha/2} u(x_j) - (-\Delta_h)^{\alpha/2}u_j =
    \frac{1}{2\pi}\left[\int_{-\infty}^{\infty}\! |\xi|^{\alpha} e^{i\xi x_j}\hat{u}(\xi) \,\mathrm{d}\xi-
\int_{-\pi/h}^{\pi/h} \!\tilde{M}_h(\xi)e^{i\xi x_j}\hat{u}(\xi) \,\mathrm{d}\xi
\right].
\]
If $\hat{u}(\xi)$ decays to zero fast enough,  then the first integral above
is essentially determined on the interval $[-\pi/h,\pi/h]$,
or
\[
    (-\Delta)^{\alpha/2} u(x_j) - (-\Delta_h)^{\alpha/2}u_j \approx
    \frac{1}{2\pi}\int_{-\pi/h}^{\pi/h} \Big(|\xi|^\alpha-\tilde{M}_h(\xi)\Big)e^{i\xi x_j}\hat{u}(\xi)
\,\mathrm{d}\xi.
\]
If the rescaled symbol can be expanded as $M(\xi) = |\xi|^\alpha\big(1+a_1 \xi
+a_2\xi^2+\cdots\big)$ near the origin, or equivalently $\tilde{M}_h(\xi) = |\xi|^{\alpha}
\big(1+a_1\xi^2h^2+a_2\xi^2h^2+\cdots\big)$, the above error becomes
\[
(-\Delta)^{\alpha/2} u(x_j) - (-\Delta_h)^{\alpha/2}u_j \approx
C_1 h + C_2 h^2 + C_3h^3+\cdots,
\]
where the constants
\[
  C_1  = -\frac{a_1}{2\pi}\int_{-\pi/h}^{\pi/h} \xi |\xi|^\alpha e^{i\xi x_j}\hat{u}(\xi) \dr\xi,\quad
  C_2 = -\frac{a_2}{2\pi}
  \int_{-\pi/h}^{\pi/h} |\xi|^{\alpha+2} e^{i\xi x_j}\hat{u}(\xi) \dr\xi, \quad \cdots
\]
are bounded. If $a_k$ (hence $C_k$) is the first non-zero coefficient
in the above expansion of $M(\xi)$, the lead order error is $O(h^{k})$, which is the order of accuracy
of the scheme.

Now we can examine the order of accuracy from $M(\xi)$.
Since $M^{SP}(\xi)$ is exactly $|\xi|^{\alpha}$,
the scheme is spectrally accurate and the error is exactly the
integral $\int_{|\xi|>\pi/h} |\xi|^{\alpha}e^{i\xi x_j}\hat{u}(\xi)
\,\mathrm{d}\xi$.
For the regularized symbol $M^{SP}(\xi) =
\left[2\big(1-\cos(\xi)\big)\right]^{\alpha/2}$, the expansion
\[
M^{PER}(\xi)= |\xi|^\alpha\left(1-\frac{\alpha}{24}\xi^2+\left[\frac{\alpha^2}{1152}
-\frac{\alpha}{2880}\right]\xi^4+\cdots\right),
\]
implies second order accuracy.
For $M^{GL}(\xi)$ motivated from fractional derivatives,
from the simplified expression~\eqref{eq:GLM}, we have
\[
 M^{GL}(\xi) = \begin{cases}
  |\xi|^{\alpha}\left[
  1+\frac{\alpha}{2}\tan \left(\frac{\pi\alpha}{2}\right)|\xi|
  -\left(\frac{\alpha^2}{8}+\frac{\alpha}{12}\right)\xi^2+\cdots\right],\qquad & \alpha \in (0,1) \cr
  |\xi|^{\alpha}\left[
  1+\left(\frac{\alpha}{2}-1\right)\tan\left(\frac{\alpha\pi}{2}\right) |\xi|
  -\left(\frac{\alpha^2}{8}-\frac{5\alpha}{12}+\frac{1}{2}
  \right)\xi^2+\cdots
  \right], &\alpha \in (1,2),
 \end{cases}
\]
which implies only first order accuracy. Finally for the weights $w^{T}$ or $w^Q$ from quadrature, 
although no explicit symbol $M(\xi)$ available,
the accuracy is proved in~\cite{HuangOberman1} to be $h^{2-\alpha}$
for $w^T$ and $h^{3-\alpha}$ for $w^Q$ (seems $h^{4-\alpha}$ for $\alpha$ between $1$ and $2$).
The accuracy of these schemes are verified in Figure~\ref{fig:weightconv}, for the fractional Laplacian of $u(x)=e^{-x^2}$ at the origin, while the spectral convergence of $w^{SP}$
is already confirmed in Figure~\ref{fig:spectralconv}.

\begin{figure}[htp]
 \begin{center}
  \includegraphics[totalheight=0.27\textheight]{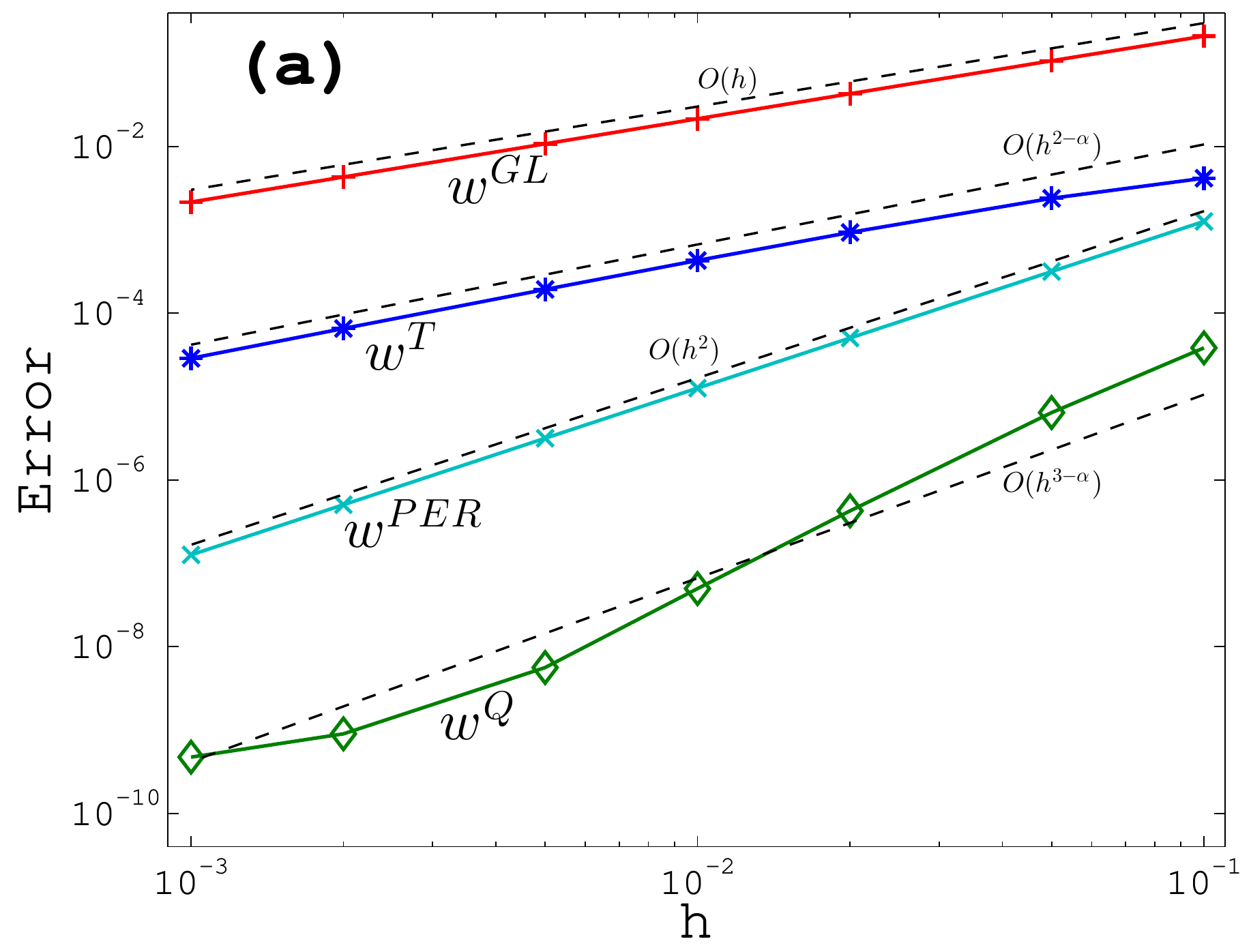}$~$
  \includegraphics[totalheight=0.27\textheight]{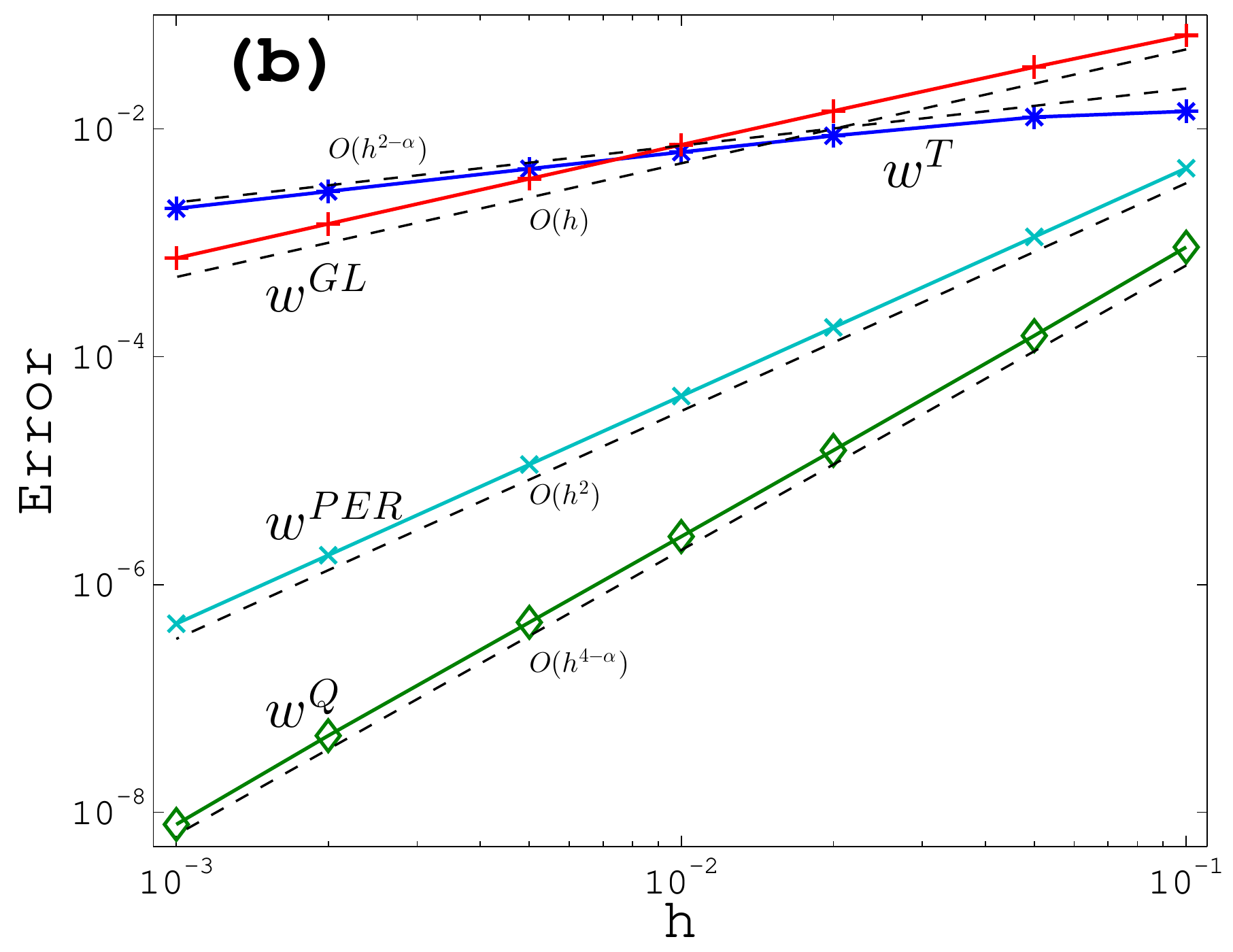}
 \end{center}
 \caption{Convergence rates of different weights in the scheme~\eqref{FLh} for
     the fractional Laplacian of $u(x)=e^{-x^2}$ at the origin:
     {\bf (a)} $\alpha = 0.8$ and {\bf (b)} $\alpha = 1.5$.}
\label{fig:weightconv}
\end{figure}

\begin{rem}
    In practice, the discretisation error depends on other factors.
    If $u$ is not smooth enough, $\hat{u}(\xi)$ does not decay to
    zero fast enough. As a result, the error may be dominated by
    the integral $\int_{|\xi|>\pi/h} |\xi|^{\alpha}e^{i\xi x_j}\hat{u}(\xi)d\xi$,
    exhibiting a different rate of convergence. Moreover,
    when the infinite sum in the scheme~\eqref{FLh} is truncated,
    the desired accuracy may be lost due to
    inappropriate treatment of the boundary conditions.
\end{rem}

\subsection{Summary on the properties of the schemes}
Now we can summarise the properties of different schemes,
as shown in Table~\ref{tab:schemeprop}, and provide some guidance on
the right one to implement in practice. The  \GL scheme,
despite its popularity due to the connection with fractional derivatives,
has only linear convergence rate and the time step in explicit method for evolution
equations could be very restricted at $\alpha\approx 1$.
The weight $w^{T}$ from quadrature with piecewise linear functions has an order of
accuracy less than linear for $\alpha>1$. Therefore,
$w^{GL}$ and $w^{T}$ are less favoured than the rest schemes.
In contrast, both the schemes with $w^{PER}$ and $w^{Q}$
are accurate in most applications, with all the desired properties.
Finally, the spectral method with $w^{SP}$ is very accurate for smooth and
fast decaying functions. However the corresponding weights can be negative
for $\alpha>1$ and do not reduce to the three-point central difference
when $\alpha$ goes to $2$. Therefore, this spectral scheme with $w^{SP}$
may not always be the best choice.

\begin{table}
\begin{center}
\begin{tabular}{  c | c    p{5.1cm} c c }
\hline
& Non-negativity& \qquad CFL Condition ($C_{\mbox{max}}$)  & Accuracy
& limit as $\alpha\to 2^-$ \\  \hline
$w^{SP}$ &  for $0<\alpha \leq 1$ &\qquad  $\pi^{-\alpha}(1+\alpha)$ & Spectral & No\\[1em]
$w^{PER}$  & Yes &  $\qquad
\frac{\alpha\Gamma(\alpha/2)^2}{4\Gamma(\alpha)}$ & $O(h^2)$ & Yes \\[1em]
$w^{GL}$ & Yes & \quad $\cos(\alpha\pi/2),\quad \alpha\in(0,1)$
                  & $O(h)$ & Yes \\
&              &    $-\cos(\alpha\pi/2)/\alpha,\quad \alpha\in (1,2)$       & &
                  \\[0.6em]
$w^T$ & Yes & \quad \   $ \dfrac{\pi^{1/2}\Gamma(2-\alpha/2)}
{2^\alpha\Gamma(\frac{1+\alpha}{2})}$ & $O(h^{2-\alpha})$     &Yes  \\[0.6em]
$w^Q$ & Yes & \quad \    $ \dfrac{\pi^{1/2}\Gamma(2-\alpha/2)}
{2^\alpha\Gamma(\frac{1+\alpha}{2})}$ &$O(h^{3-\alpha})$ &Yes
\\ \hline
\end{tabular}
\end{center}
\caption{A summary of properties of the weights: the non-negativity (except $w_0$),
    the CFL condition $C_{\mbox{max}}=\big(-h^{\alpha}w_0\big)^{-1}$, the
order of accuracy and the convergence to the three-point standard scheme as $\alpha \to 2^-$.
}
\label{tab:schemeprop}
\end{table}

\section{Truncation of the far-field boundary conditions}\label{sec:truncation}

The general scheme~\eqref{FLh} involves infinitely many terms, which has to be
truncated in practice. For Dirichlet problems on a bounded domain,
the infinite sum can usually be approximated using the given boundary conditions
on the exterior of the domain. For problems posed on the whole space, with
zero or other reasonable boundary conditions at infinity, the situation is more
complicated, as a result of the slow decay of the solutions commonly observed in various applications.
For example,  in contrast to compactly supported solutions of the classical porous medium
equation $u_t - \Delta u^m=0$ for $m>1$, the solutions of the fractional porous medium
equation $u_t + (-\Delta)^{\alpha/2}u^m=0$ in $\mathbb{R}^n$ has an algebraic tail
of order $|x|^{-N-\alpha}$ for $\alpha \in (0,2)$, as shown in~\cite{vazquez2014}.
Because of this slow decay to zero, a straightforward truncation like
$\sum_{k=-M}^{M}(u_j-u_{k})w_{j-k}$ of~\eqref{FLh} may lead to significant errors
even for relative large value of $M$.

A first step towards truncating the asymptotic far-field boundary condition was given
by the authors in~\cite{HuangOberman1}, and is reviewed briefly below. Without loss of
generality, the computational domain is denoted by $[-L,L]$, on which the fractional Laplacian
of $u$ at $x_j$ is sought. To proceed, first the singular integral~\eqref{eq:rieszfrac}
is decomposed into three parts,
\begin{multline} \label{eq:threeparts}
 (-\Delta)^{\alpha/2}u(x_j) =
\underbrace{C_{1,\alpha}\int_{-L_M}^{L_M}\!
\big(u(x_j)-u(y)\big)|x_j-y|^{-1-\alpha} \, \mathrm{d}y}_\textrm{(I)} \cr
+ \underbrace{C_{1,\alpha}u(x_j)\int_{|y|>L_M}\! |x_j-y|^{-1-\alpha}
\,\mathrm{d} y}_\textrm{(II)}
- \underbrace{C_{1,\alpha}\int_{|y|>L_M}\! u(y)|x_j-y|^{-1-\alpha}
\,\mathrm{d}y}_\textrm{(III)},
\end{multline}
for some  $L_M \geq L$.
The first term $\textrm{(I)}$ is approximated by the finite truncation
$\sum_{k=-M}^{M}(u_j-u_{k})w_{j-k}$, and
the integral in the second term $\textrm{(II)}$ can be evaluated exactly,
while the last term $\textrm{(III)}$ has to be estimated using asymptotic
far-field boundary conditions. If $u$ decays to zero with an algebraic rate $\beta$,
that is $u(x)=O(|x|^{-\beta})$ for some $\beta>0$, then
\begin{equation}\label{eq:asyuy}
 u(y) \approx u(\pm L)L^\beta |y|^{-\beta}, \qquad y \to \pm \infty.
\end{equation}
Hence the last term $\textrm{(III)}$ can be estimated using the following two integrals,
{\small
\begin{multline}\label{eq:BCM}
C_{1,\alpha}\int_{L_M}^\infty\! u(y)|x_j-y|^{-1-\alpha}\,\mathrm{d}y\approx
C_{1,\alpha}u(L)L^\beta \int_{L_M}^\infty\! (y-x_j)^{-1-\alpha}y^{-\beta}\, \mathrm{d}y\cr
=\frac{C_{1,\alpha}u(L)L^\beta}{ (\al+\beta)(L_M)^{\alpha+\beta}}\ {}_2F_1\Big(\alpha+1,\al+\beta;\al+\beta+1;\frac{x_j}{L_M}
\Big),
\end{multline}
}
and
{\small
\begin{multline}\label{eq:BCP}
 C_{1,\alpha}\int^{-L_M}_{-\infty}\! u(y)|x_j-y|^{-1-\alpha}\,\mathrm{d}y \approx
C_{1,\alpha}u(-L)L^\beta
\int^{-L_M}_{-\infty}\! (x_j-y)^{-1-\alpha}(-y)^{-\beta}\, \mathrm{d}y\cr
=\frac{C_{1,\alpha}u(-L)L^\beta}{ (\al+\beta)(L_M)^{\alpha+\beta}}\ {}_2F_1\Big(\alpha+1,\al+\beta;\al+\beta+1;-\frac{x_j}{L_M}\Big),
\end{multline}
}
in terms of the Gauss hypergeometric function ${}_2F_1$.

\begin{rem} We add a few more comments related to the practical implementation
  of this approach.
If the exact value of the exponent $\beta$ is not available from existing theory,
it can still be estimated from the solution itself by fitting the computed data.
The extended domain size $L_M$ should be strictly larger than $L$
($L_M=3L$ in~\cite{HuangOberman1}), to avoid any possible singularity
of ${}_2F_1(\alpha+1,\al+\beta;\al+\beta+1;z)$ at $z=1$.
In this way, if $x_k$ is outside the interval $[-L,L]$,
$u_k$ can be approximated by~\eqref{eq:asyuy}
in the computation of the finite sum $\sum_{k=-M}^{M}(u_j-u_{k})w_{j-k}$.
\end{rem}

The above treatment of the far-field boundary condition is natural in the context
of the difference-quadrature scheme derived in~\cite{HuangOberman1},
based on the singular integral~\eqref{eq:rieszfrac}.
It turns out that the same approach works for a wide variety of schemes,
because of the ubiquitous scaling $w_k \approx C_{1,\alpha}h^{-\alpha}k^{-1-\alpha}$
shared by most of the weights (see Section~\ref{sec:weightprop}).
This scaling behaviour is expected from a careful comparison between the singular
integral~\eqref{eq:rieszfrac} and the scheme~\eqref{FLh}, that is,
\[
 w_{k} \approx C_{1,\alpha}\int_{x_{k}-h/2}^{x_{k}+h/2}\!
|y|^{-1-\alpha}\,\mathrm{d}y
\approx C_{1,\alpha}\int_{x_k-h/2}^{x_k+h/2} |x_k|^{-1-\alpha} \dr y
= C_{1,\alpha} h^{-\alpha}k^{-1-\alpha}.
\]
Therefore, if $u(y)$ is approximated by the constant $u(x_{k})$ for $y \in (x_{k}-h/2,x_{k}+h/2)$,
\begin{equation}\label{eq:farfieldapprox}
    (u_{j}-u_{k})w_{j-k} \approx  C_{1,\alpha}
    \int_{x_{k}-h/2}^{x_{k}+h/2}\!
    |x_j-y|^{-1-\alpha}\big(u(x_j)-u(y)\big) \, \mathrm{d}y,
  \end{equation}
and the infinite sum in the scheme~\eqref{FLh} can be written as
\begin{multline}\label{eq:FFBC}
 \sum_{k=-M}^{M} (u_j-u_{k})w_{j-k}+
\sum_{k>|M|}(u_j-u_{k})w_{j-k} \approx \sum_{k=-M}^{M} (u_j-u_{k})w_{j-k} \cr
+ C_{1,\alpha}u_j\int_{|y|>x_{M}+h/2}\! |x_j-y|^{-1-\alpha}\,\mathrm{d}y
- C_{1,\alpha}\int_{|y|>x_{M}+h/2}\! u(y)|x_j-y|^{-1-\alpha}\,\mathrm{d}y.
\end{multline}
Since the last two integrals are exactly the same as those in~\eqref{eq:threeparts}
with $L_M = x_M + h/2$, the far-field boundary  conditions with other weights
can be handled in the same way as that in~\cite{HuangOberman1}.

\begin{figure}[htp]
 \begin{center}
\includegraphics[totalheight=0.27\textheight]{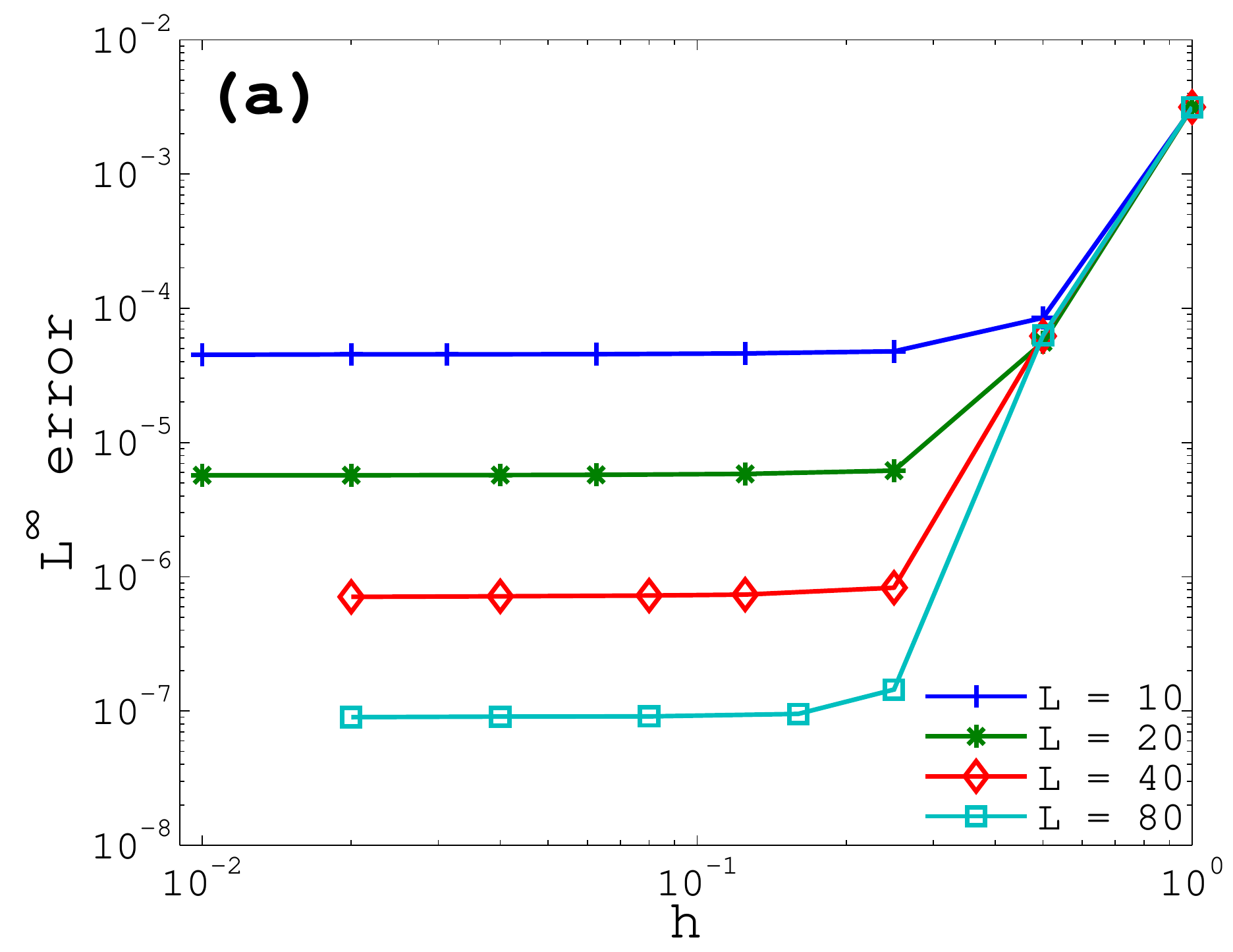}$~ ~$
  \includegraphics[totalheight=0.27\textheight]{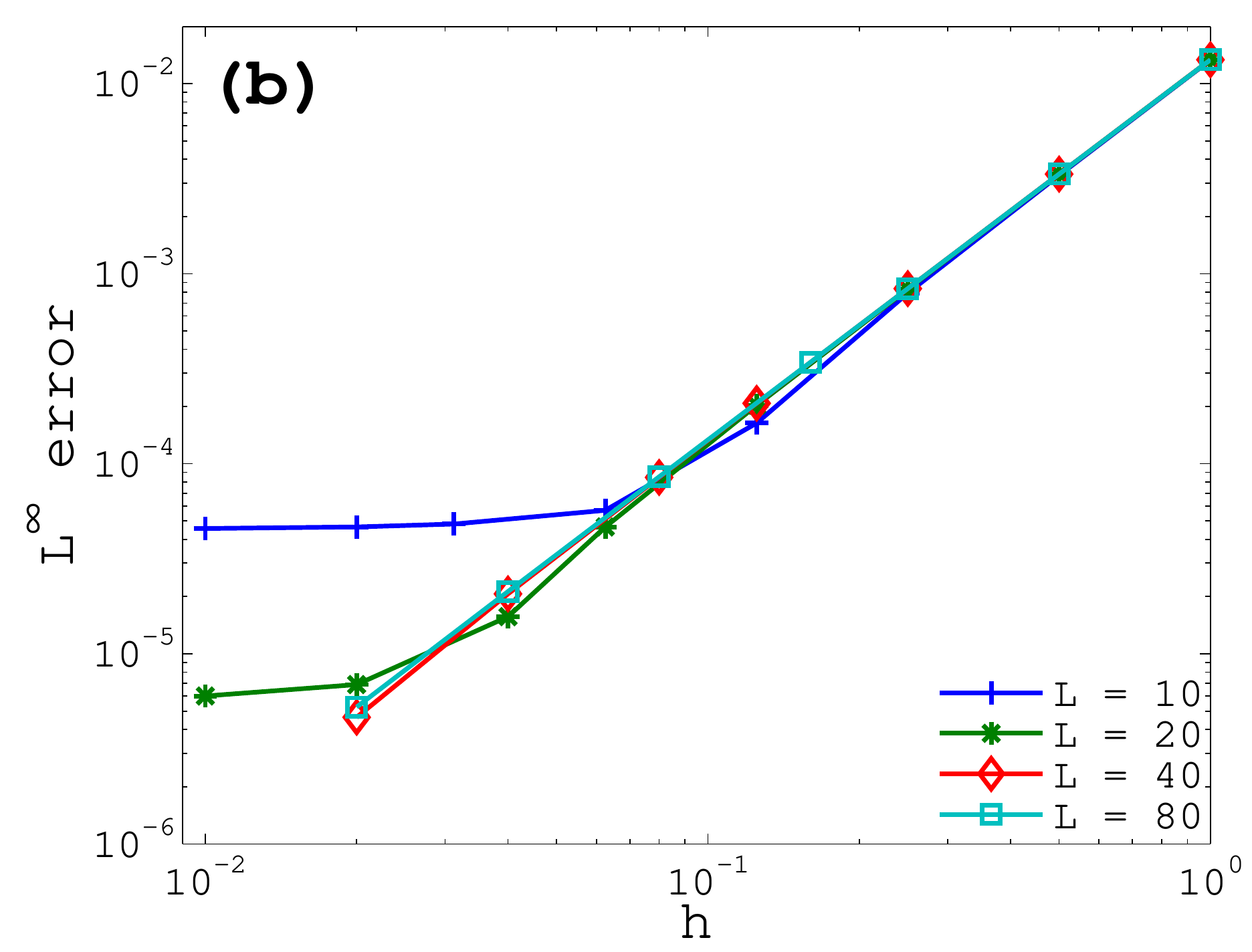}
 \end{center}
\caption{The $L^\infty$ error of the fractional Laplacian of
  $u(x) = (1+x^2)^{-(1-\al)/2}$ with $\alpha=0.4$
using the approximation~\eqref{eq:FFBC}: {\bf (a)}
the spectral weights $w^{SP}$; {\bf (b)} the weights $w^{PER}$.}
\label{fig:FFBC}
\end{figure}

The effectiveness of this the far-field boundary condition is shown
in Figure~\ref{fig:FFBC} on different domain sizes $L$ and grid
sizes $h$, for the function $u(x) = (1+x^2)^{-(1-\al)/2}$ and its Fractional Laplacian
\begin{equation} \label{eq:onedfractest}
(-\Delta)^{\alpha/2} u(x)
=2^\al \Gamma\Big(\frac{1+\al}{2}\Big)\Gamma\Big(\frac{1-\al}{2}\Big)^{-1}(1+x^2)^{-(1+\al)/2}.
\end{equation}
The convergence rates for $w^{SP}$ and $w^{PER}$ shown  in Figure~\ref{fig:FFBC} are very similar
to those in~\cite{HuangOberman1} for $w^Q$: as the grid size $h$ decreases,
the error decreases with the expected rate as summarized in Table~\ref{tab:schemeprop},
then it levels off because its dominant contribution is taken over by
the boundary condition.

\begin{rem}
The error  for $w^{SP}$ seems to be ``spectral'' before saturation, when
$alpha$ is between $0$ and $1$ as in Figure~\ref{fig:FFBC}.
However when $\alpha$ is between $1$ and $2$,
$w_k^{SP}$ scales like $k^{-2}h^{-\alpha}$  with alternating signs.
The approximation~\eqref{eq:farfieldapprox} is no long valid (at least
with respect to the expected accuracy), and the above approach
to the far-field boundary condition is less effective for $w^{SP}$.

\end{rem}

\section{Convergence for equations with fractional Laplacian operators}
\label{sec:conv}

Although the efficient and accurate evaluation of the fractional Laplacian is fundamental,
its practical applications in the numerical solutions of equations with
this operator are equally important.
In this section, we study the convergence of the canonical
extended Dirichlet boundary value problem,
\begin{equation}\label{FLD}\tag{FLD}
\begin{aligned}
\FL u &= f && \text{ for } x \in  \Dd,
\\
u &= g  && \text{ for  } x \in \mathbb{R}^n\setminus\Dd.
\end{aligned}
\end{equation}
Here $\Dd$ is a bounded domain on which the unknown function $u$ is sought and
$g$ is the boundary condition.
The problem \eqref{FLD} is the analogue of the Dirichlet problem for
the Laplace operator, except that, here $g$ is given
 on the complement of the domain $\Dd$, instead of just the boundary $\partial \Dd$.
In the special case $g=0$, $f=1$, the solution $u$  corresponds to
the expected first passage time (or mean exit time) of the
symmetric  $\alpha$-stable L\'{e}vy process from a given domain~\cite{MR0137148}.
When the domain $\Dd$ is a ball about the origin and $f=0$,
the solution can be written explicitly as a potential integral of  $g$
(see \cite[Chapter I]{MR0350027} and
\cite[Section 5.1]{silvestre2007regularity} for more details), extending the
Poisson formula for the Laplace equation to the fractional setting.


Consider~\eqref{FLD} in one dimension with $D=(-1,1)$.
When the scheme~\eqref{FLh} is applied,  the discrete problem becomes
\begin{equation}\label{FLhD}\tag{$\text{FLD}_h$}
\begin{aligned}
\FLh u_j &= f_j && \text{ for } j \in  \Ddh, \\
u_j &= g_j  && \text{ for  } j \in \DdCh,
\end{aligned}
\end{equation}
where $\Ddh = \{j \in \mathbb{Z}: \abs{j h} < 1\}$, $\DdCh = \mathbb{Z}_h\setminus\Ddh$
with $f_j=f(x_j)$ and $g_j=g(x_j)$. We are interested in the
convergence of the numerical solution $u_j$ to the exact solution
$u(x_j)$ on $\Dd_h$, as the grid size $h$
decreases. When different weights are employed in the discrete operator $\FLh$,
an order $O(h^p)$ local truncation error of the residue  $r_j = (-\Delta_h)^{\alpha/2} u(x_j)-f(x_j)$
does not always imply the same order of error $e_j = u(x_j)-u_j$ for the numerical solution;
some form of stability is required. In terms of the errors $r_j$ and $e_j$, the discrete
problem~\eqref{FLhD} becomes
\begin{equation}\label{eq:re}
  (-\Delta_h)^{\alpha/2}e_j = r_j \ \mbox{ for } j \in \Dd_h, \quad
  e_j = 0\ \mbox{ for } j \in \DdCh.
\end{equation}
This can be written as a linear algebra problem $L_h\mathbf{e}_h = \mathbf{r}_h$,
for some matrix $L_h$, vectors $\mathbf{e}_h$ denoting the numerical error $e_j$ with $j\in \Dd_h$
and $\mathbf{r}_h$ denoting the consistence error $r_j$. The stability
(or the equivalence of the two errors) is essentially the boundedness of the inverse
matrix $(L_h)^{-1}$ in suitable norms.

In many cases, the norm $\|(L_h)^{-1}\|$ is difficult to estimate, as the entries of the
matrix $L_h$ depend on the particular weights in the discretization.
Alternatively, the same stability conditions can be established
by constructing discrete supersolutions and applying the discrete maximum principle,
in a similar way as  for classical elliptic equations~\cite{larsson2009partial}.
To process,we have the following lemma.

\begin{lem}[Discrete maximum principle] \label{lem:MP}
Let $u$ defined on $\mathbb{Z}_h$ satisfy $\FLh u_j \le 0$ for $j \in \Ddh$, for
any discretization~\eqref{FLh} with nonnegative weights $w_k$ for $k \neq 0$.  Then
\[
\max_{j \in \Ddh} u_j \le \max_{j \in \DdCh} u_j.
\]
Similarly, if $\FLh u_j \ge 0$ for $j \in \Ddh$, then
\[
\min_{j \in \Ddh} u_j \ge  \min_{j \in \DdCh} u_j.
\]
\end{lem}

Next we would like to construct discrete supersolutions $v$ on $\mathbb{Z}_h$, which satisfies
\begin{equation}\label{vbarrier}
\begin{aligned}
\FLh v_j &\ge  1, && \text{ for } j \in  \Ddh,
\\
v_j &\geq 0,  && \text{ for  } j \in \DdCh.
\end{aligned}
\end{equation}
Applying Lemma~\ref{lem:MP}, $v$ is automatically non-negative on $\Dd_h$.
If such a supersolution $v$ exists on $\mathbb{Z}_h$, the discrete function
$p$ defined as $p_j=\|\mathbf{r}_h\|_\infty
v_j\pm e_j$ satisfies $(-\Delta_h)^{\alpha/2}p_j \geq 0$ for $j \in \Ddh$ and
$p_j \geq 0$ for $j\in \DdCh$. Therefore, $p$ is non-negative on $\mathbb{Z}_h$, and
the numerical error $|e_j|$ is bounded by $\|\mathbf{r}_h\|_\infty\|v\|_\infty$,
showing the same order
of local truncation error $\mathbf{r}_h$ and numerical error $\mathbf{e}_h$.
For classical elliptic equations, the supersolutions (either continuous or
discrete) can usually be chosen as an inverted parabola, but it is more difficult
for the equations with the fractional Laplacian.
In~\cite{HuangOberman1}, the following discrete supersolution
\[
v_j =
\begin{cases}
4- (jh)^2, & \abs{jh} < 1\\
0, & \text{ otherwise},
\end{cases}
\]
is chosen (associated with the weights $w^{Q}$), and the conditions~\eqref{vbarrier} is verified with
some technical assumptions. When other weights are used in the scheme~\eqref{FLh},
\eqref{vbarrier} may be difficult to verify analytically for a specific choice of supersolutions.
Instead, we look for motivation from the continuous supersolution $v_G$, which satisfies
\[
(-\Delta)^{\alpha/2}v_G(x) = 1, \quad x \in (-1,1),
\]
and $v_G(x)\equiv 0$ for $|x|\geq 1$, by showing numerically that $v_{Gj} = v_G(x_j)$
is a discrete supersolution. This function $v_G$ is the first passage time from the unit
ball for particles under symmetric $\alpha$-stable process~\cite{MR0137148}
and is explicitly given by
\[
  \qquad \qquad v_G(x) = K_\alpha (1-|x|^2)^{\alpha/2}_+,\qquad
  K_\alpha = \frac{2^\alpha}{\sqrt{\pi}}\Gamma\left(1+\frac{\alpha}{2}\right)
  \Gamma\left(\frac{1+\alpha}{2}\right).
\]

\begin{figure}[htp]
  \begin{center}
    \includegraphics[totalheight=0.26\textheight]{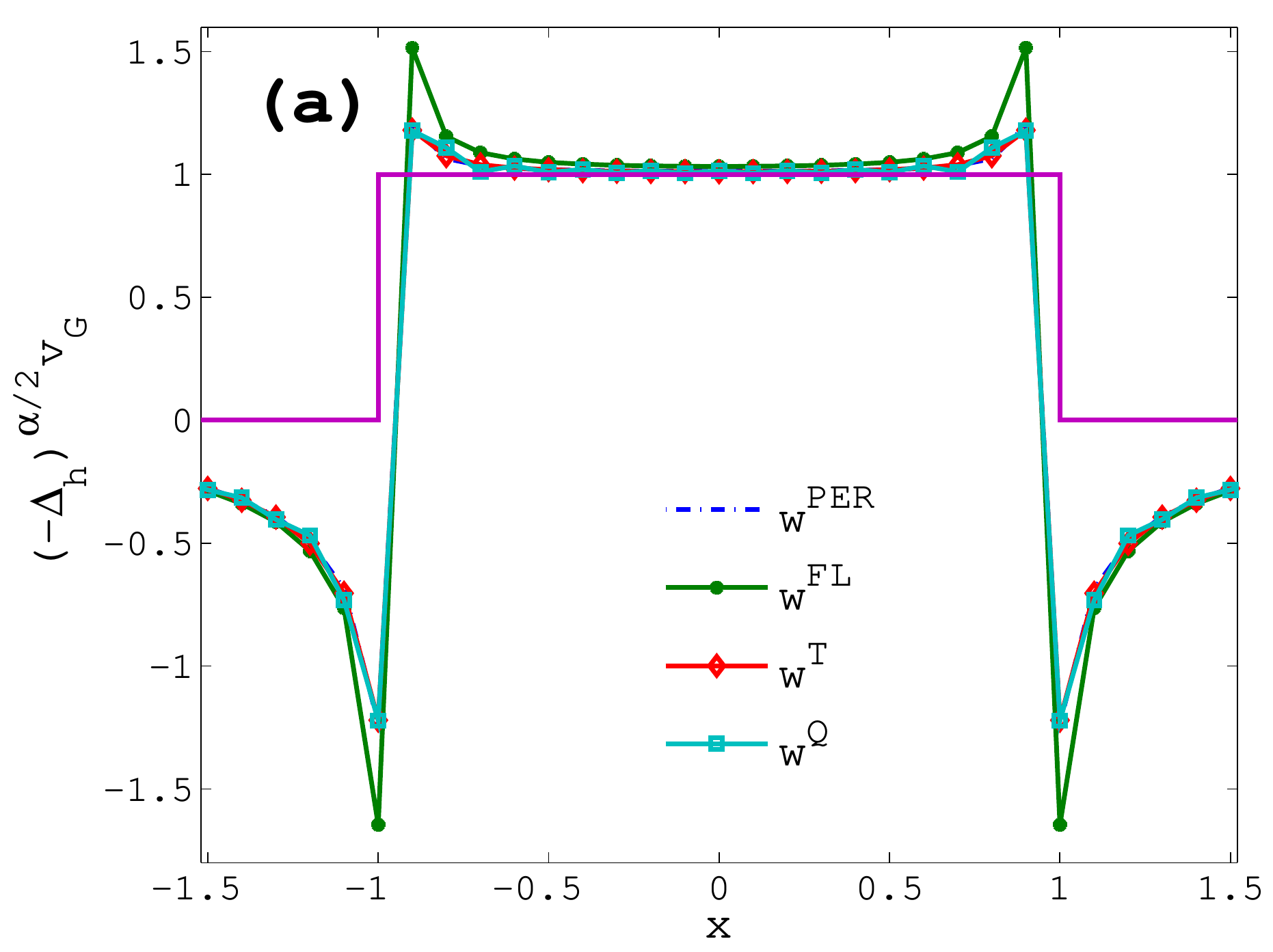}$~$
    \includegraphics[totalheight=0.26\textheight]{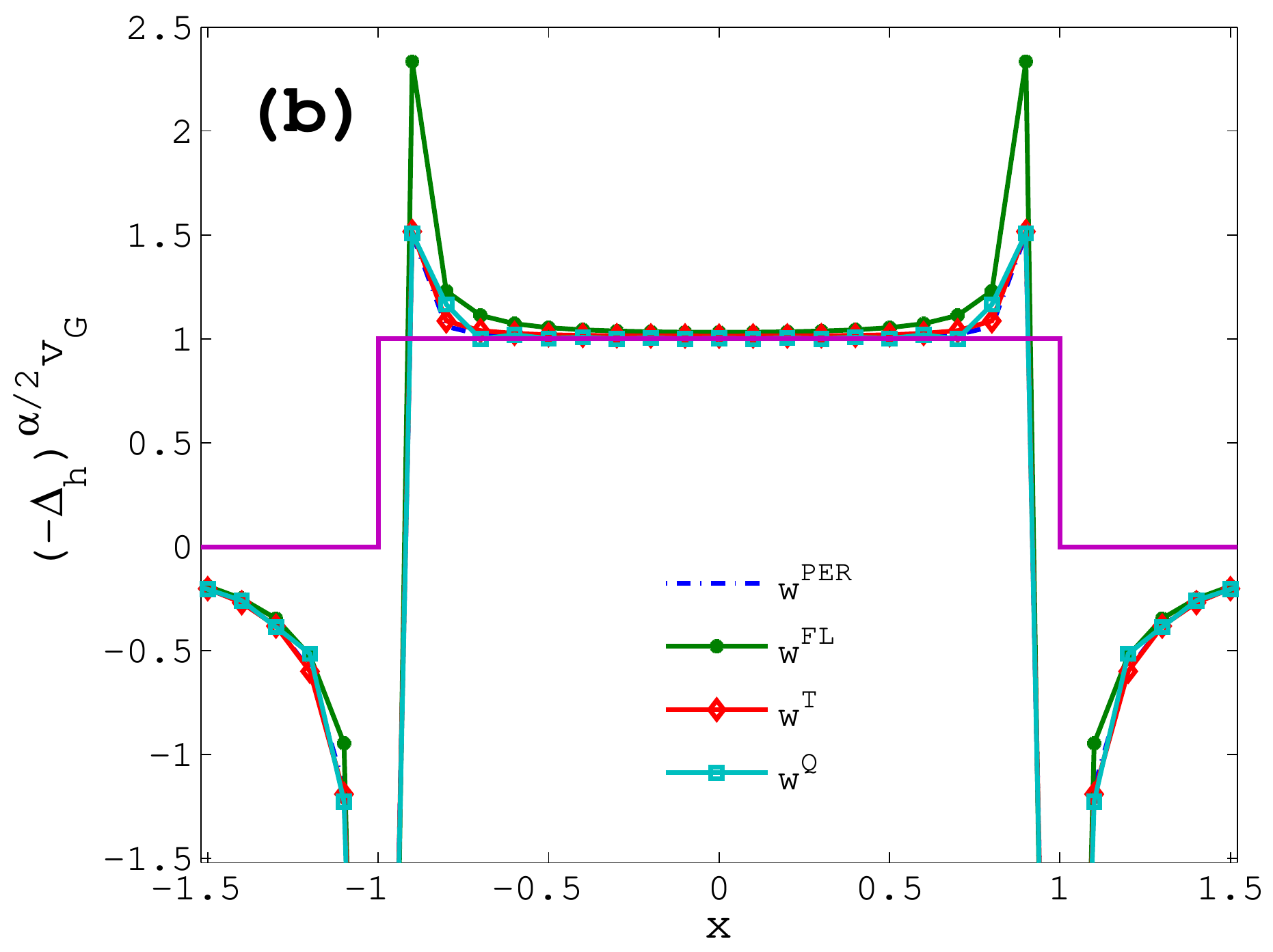}
  \end{center}
  \caption{The (discrete) fractional Laplacian of $v_G$ for various schemes:
    (a) $\alpha=0.5$  and (b) $\alpha=1.5$, suggesting that $v_G$ is a discrete
  supersolution for all $\alpha \in (0,2)$ and all non-negative weights considered in
this paper.}
  \label{fig:flvg}
\end{figure}
The discrete fractional Laplacian of $v_G$ is shown in Figigure~\ref{fig:flvg}
for various scheme with $\alpha=0.5$ and $\alpha = 1.5$. It seems that~\eqref{vbarrier}
is always satisfied, and $v_G$ is a valid supersolution with the associated weights~\footnote{
  It is important to make sure that the boundary $\pm 1$ lies also on the boundary
$\partial D_h^C$, i.e., $1/h$ is an integer.
}.

\section{Numerical experiments and applications to nonlinear PDEs}\label{sec:nonlinear}

In this section, we provide several numerical examples, to further verify the
accuracy of the scheme with different weights, especially the relation between
the order of convergence and the regularity of the functions,
and to showcase the applications in various PDEs.

\subsection{Accuracy when the solutions are non-smooth}
The order of accuracy of the schemes derived in
Section~\ref{sec:weightsFourier} is valid only for smooth functions,
such that the integral $\int_{|\xi|>\pi/h} |\xi|^{\alpha} e^{i\xi x_j}
\hat{u}(\xi)\mathrm{d}x$ can be safely ignored. For less smooth
functions, the actual convergence rate can be lower as we show now.
The examples used in this and the next subsection re related to the following result~\cite{MR3294409,MR3239623}
\begin{equation}\label{eq:expfraceq}
 (-\Delta)^{\frac{\alpha}{2}}(1-x^2)_+^{k+\frac{\alpha}{2}}
=\begin{cases}
  K_{k,\alpha}\ {}_2F_1\left(\frac{1+\alpha}{2},-k;\frac{1}{2};x^2\right),\qquad & |x|<1,\cr
\tilde{K}_{k,\alpha}\ {}_2F_1\left(\frac{1+\alpha}{2},\frac{2+\alpha}{2};\frac{3+\alpha}{2}+k;\frac{1}{x^2}\right), & |x|>1,
 \end{cases}
\end{equation}
where $k$ is an integer,
\[
 K_{k,\alpha} = \frac{2^\alpha \Gamma\left(k+1+\frac{\alpha}{2}\right)\Gamma\left(\frac{1+\alpha}{2}\right)}
{k!\Gamma\left(\frac{1}{2}\right)},\quad
\tilde{K}_{k,\alpha} =
\frac{2^\alpha \Gamma\left(k+1+\frac{\alpha}{2}\right)\Gamma\left(\frac{1+\alpha}{2}\right)}
{\Gamma(-\frac{\alpha}{2})\Gamma(\frac{3+\alpha}{2}+k)},
\]
and ${}_2F_1$ is the Gauss hypergeometric function~\cite{MR1225604}.
Since $(1-x^2)_+^{k+\frac{\alpha}{2}}$ is supported on the interval $[-1,1]$
for any $k\geq 0$ and $\alpha \in (0,2)$, the sum in the scheme~\eqref{FLh}
has only finite number of terms and can be truncated appropriately.

\begin{figure}[htp]
\begin{center}
  \includegraphics[totalheight=0.27\textheight]{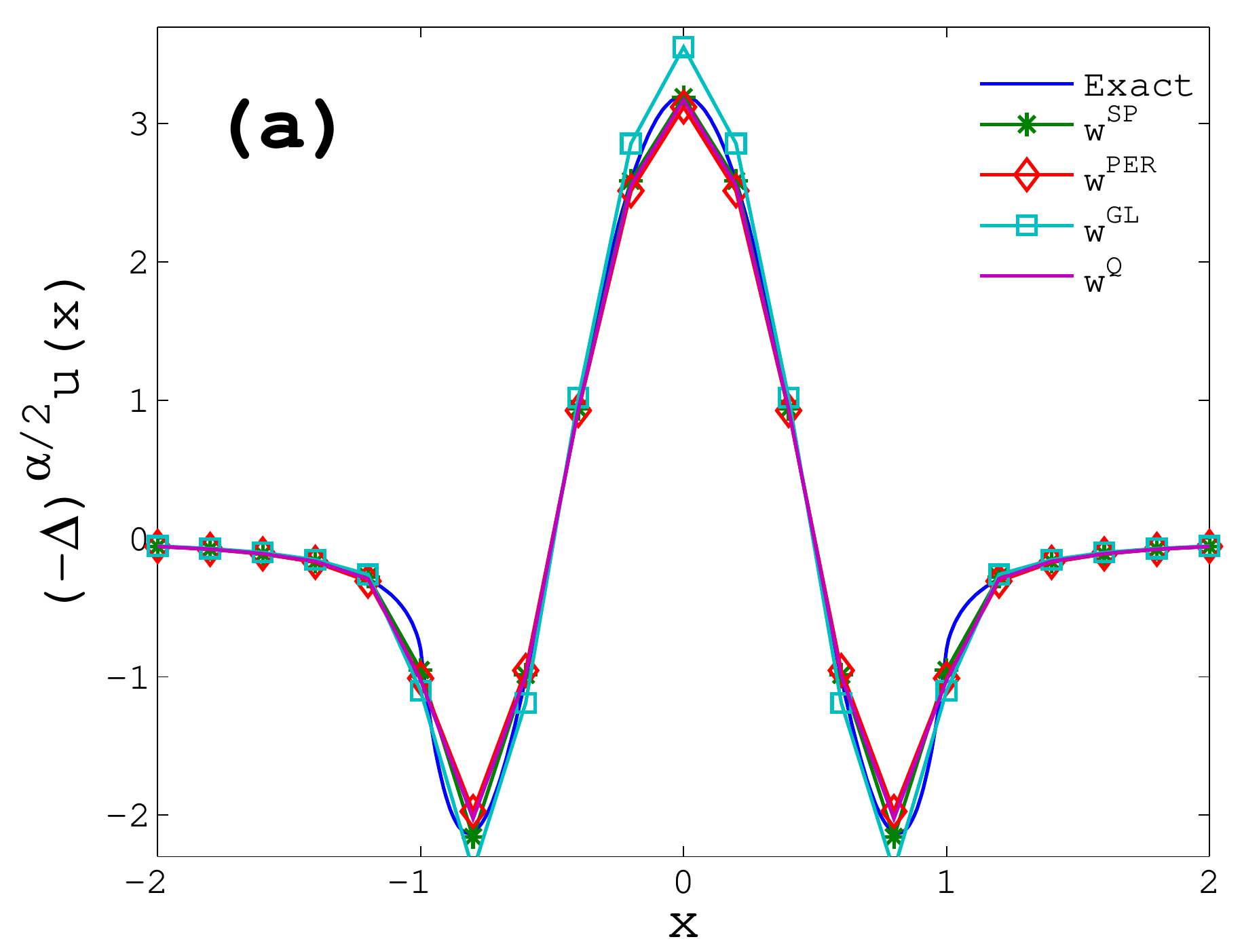}
  $~~$
  \includegraphics[totalheight=0.27\textheight]{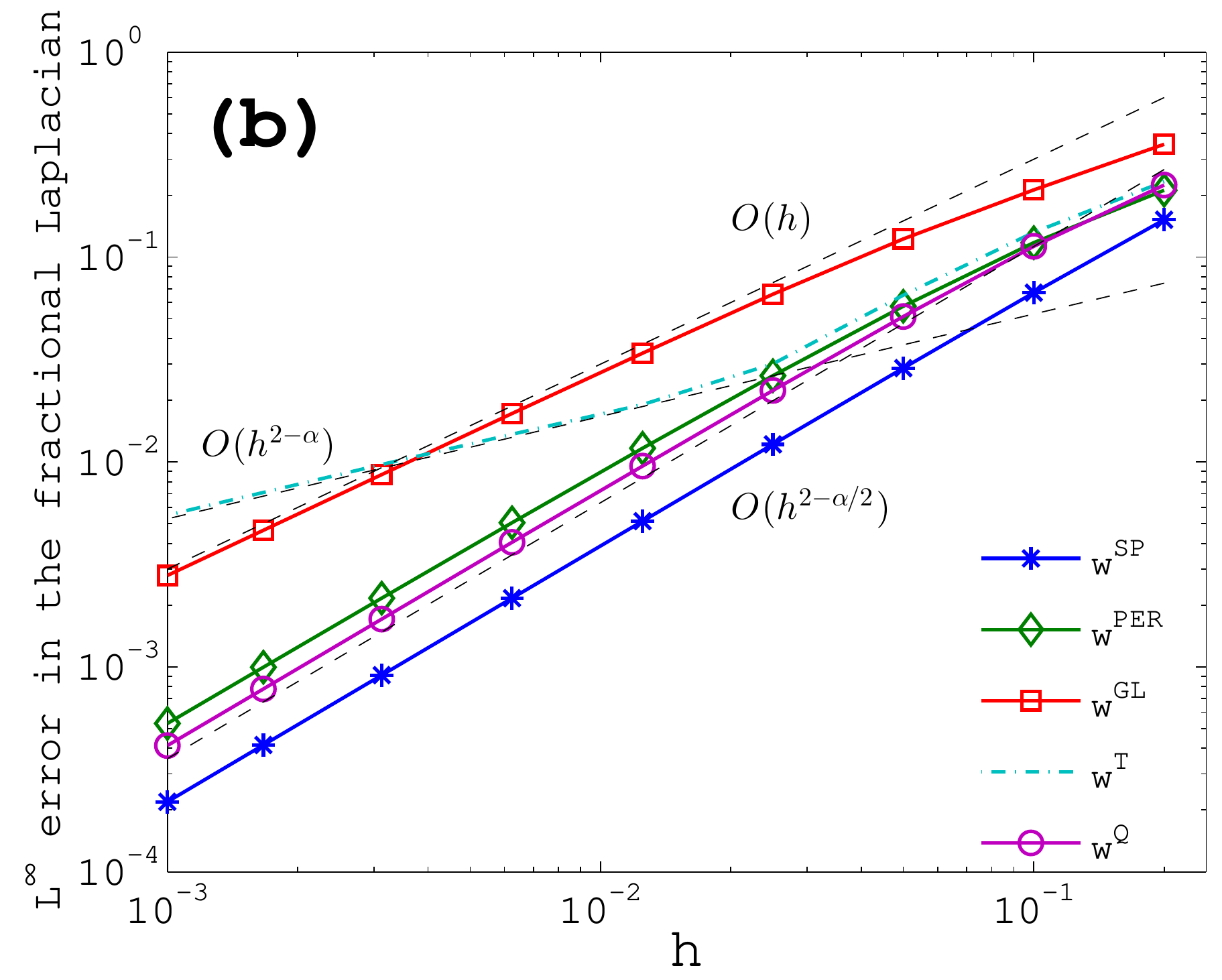}
\end{center}
\caption{ The fractional Laplacian of $u(x)=(1-x^2)_+^{2+\alpha/2}$
with $\alpha=1.5$: {\bf (a)} the result a grid with spacing $h=0.2$;
{\bf (b)} the order of convergence (measured in $L^\infty$ norm) with different $h$.}
\label{fig:redord}
\end{figure}

The fractional Laplacian of $u(x) = (1-x^2)^{2+\alpha/2}_+$ is computed with different weights,
and the results with a grid size $h=0.2$ and the converge rates are
shown in Figure~\ref{fig:redord}. While lower order schemes with
weights $w^{GL}$ and $w^{T}$ exhibit the expected order of convergence derived
in Section~\ref{sec:weightsFourier}, higher order schemes with other weights
have the same less optimal convergence rate $O(h^{2-\alpha/2})$. This
rate can be explained from the asymptotic expansion of
$\int_{|\xi|>\pi/h}|\xi|^{\alpha} e^{i\xi x_j}\hat{u}(\xi)\mathrm{d}\xi$,
the source of the dominant error. In fact, the
Fourier transform of $u(x) = (1-x^2)_+^{k+\alpha/2}$ can be explicitly expressed as
\[
  \hat{u}(\xi)=  \int_{-\infty}^\infty (1-x^2)_+^{k+\alpha/2}e^{-ix\xi}\mathrm{d}x=
  F_{k,\alpha}|\xi|^{-k-(\alpha+1)/2}J_{k+(\alpha+1)/2}(|\xi|),
\]
in terms of the Bessel function $J_\nu$ and the constant factor
\[
  F_{k,\alpha} = \sqrt{\pi}\Gamma\left(k+1+\frac{\alpha}{2}\right)2^{k+(\alpha+1)/2}.
\]
At $x_j=0$, the integral $\int_{|\xi|>\pi/h}|\xi|^{\alpha} e^{i\xi x_j}\hat{u}(\xi)d\xi$
is bounded by
\[
  F_{k,\alpha}\int_{|\xi|>\pi/h}
    |\xi|^{-k+(\alpha-1)/2}\left|J_{k+(\alpha+1)/2}(|\xi|)\right|\mathrm{d}x
    \sim  \int_{|\xi|>\pi/h} |\xi|^{-k+\alpha/2-1}\mathrm{d}\xi
    =O(h^{k-\alpha/2}),
\]
using asymptotic form $J_\mu(|\xi|)\sim \sqrt{\frac{2}{\pi |\xi|}} \cos (|\xi|-\frac{\alpha\pi}{2}
-\frac{\pi}{4})$ as $|\xi|$ goes to infinity.
The dominant order $O(h^{2-\alpha/2})$ error in Figure~\ref{fig:redord}(b)
corresponds to the choice $k=2$ in the function $u(x) = (1-x^2)^{2+\alpha/2}$.
As a general rule, if a function $u$ is
in H\"{o}lder space $C^{k,\beta}$ with optimal exponent $\beta$,
then its discrete fractional Laplacian $(-\Delta)^{\alpha/2}u$
has an order of accuracy at most $k+\beta-\alpha$.

\subsection{The extended Dirichlet problem}
We now check the accuracy of the scheme
when it is applied to the extended Dirichlet problem~\eqref{FLD},
by choosing $D=(-1,1)$, $f(x) = K_{k,\alpha}\ {}_2F_1\left(\frac{1+\alpha}{2},-k;
\frac{1}{2};x^2\right)$ on $D$ and $g \equiv 0$ on $D^C$. From~\eqref{eq:expfraceq},
the exact solution is $u(x) = (1-x^2)_+^{k+\alpha/2}$. The convergence rates
with various weights are shown in Figure~\ref{fig:extDiri}, for $k=1$ and $k=3$
respectively. Because the solution is not smooth near the boundary $x=\pm 1$,
the order of accuracy could be lower than the theoretical one. In general,
if the solution is in the H\"{o}lder space $C^{k,\beta}$ with
optimal exponent $\beta$, then the optimal order of convergence is $k+\beta$,
even if a more accuracy scheme is used.

\begin{figure}[htp]
 \begin{center}
\includegraphics[totalheight=0.27\textheight]{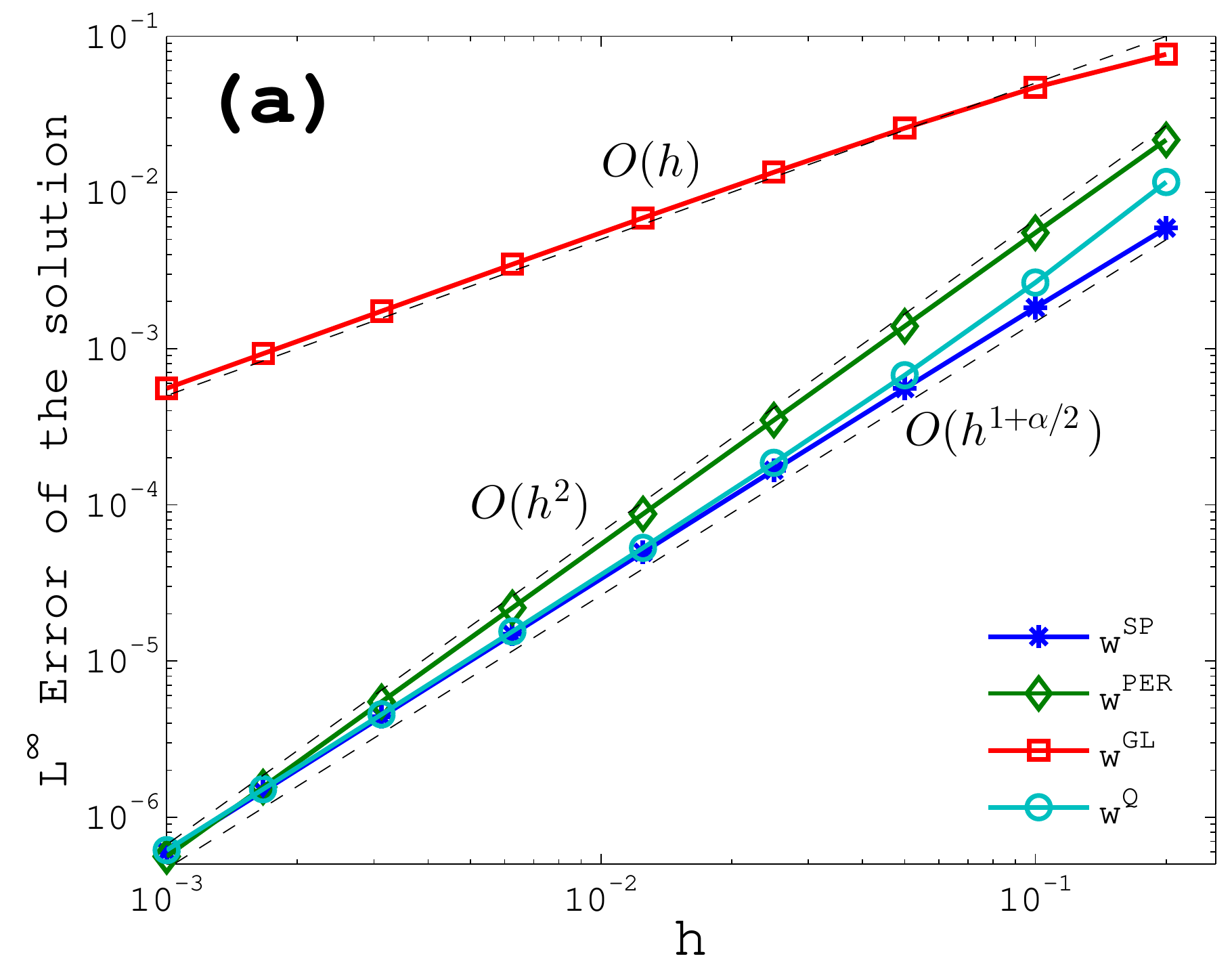}
$~ ~$
\includegraphics[totalheight=0.27\textheight]{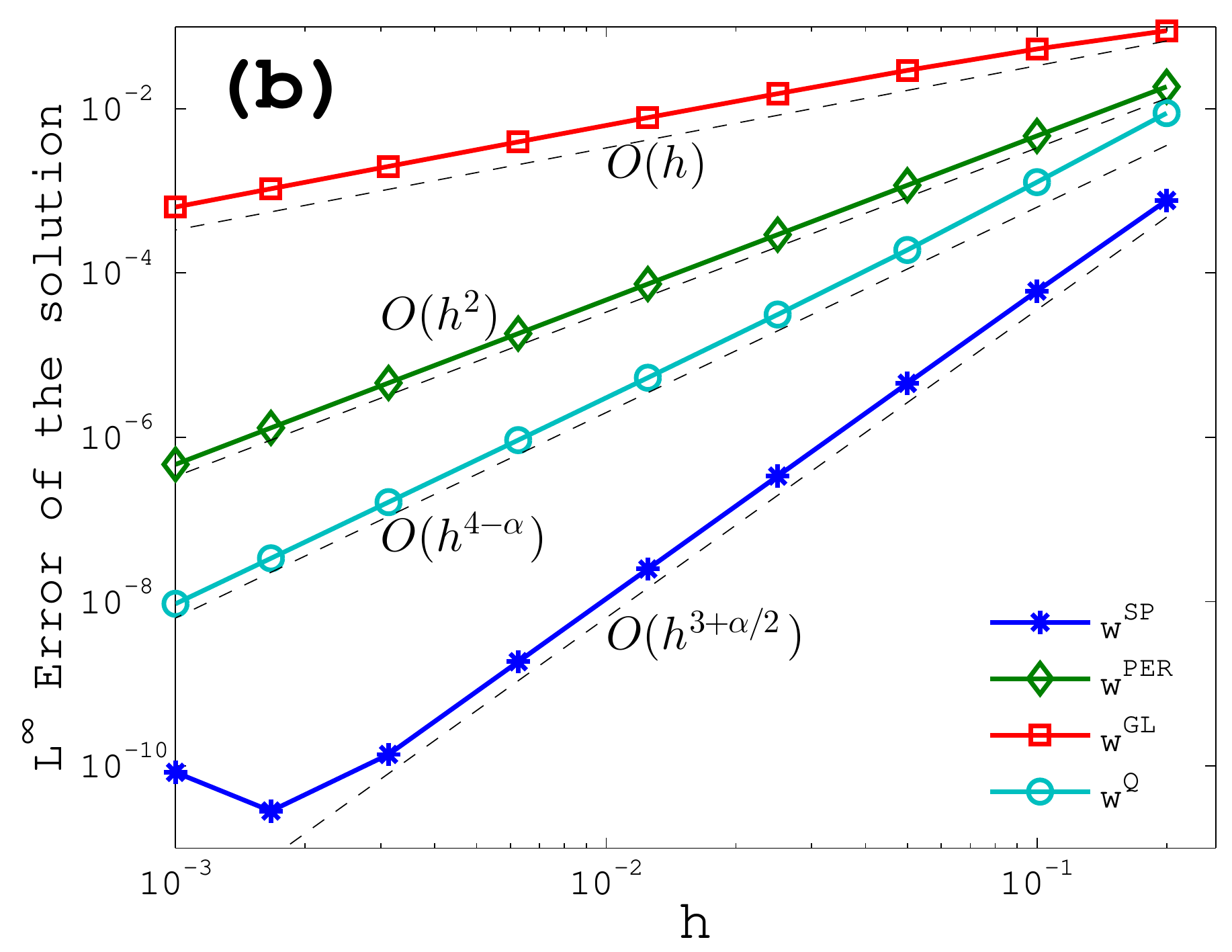}
 \end{center}
\caption{The convergence of the solution (measured in $L^\infty$
    error) for the extended  Dirichlet problem~\eqref{FLD} corresponding to the
    relation~\eqref{eq:expfraceq} (with $\alpha=1.5$): {\bf (a)} $k=1$ and {\bf (b)} $k=3$.
Because of the non-smoothness of the solution,
the order of accuracy is at most $k+\alpha/2$, even for higher order schemes.
  }
\label{fig:extDiri}
\end{figure}

\subsection{Fractional heat equation} We now move to the numerical solutions of
evolution equations with the fractional Laplacian operator, starting from the simplest
fractional  heat equation
\begin{equation}\label{eq:fracheat}
 u_t +\FL u=0,\qquad u(x,0) = u_0(x).
\end{equation}
Here $u$ is usually the probability distribution function
of particles undergoing $\alpha$-stable processes,
analogous to that in the heat equation governed by Brownian motion.
The solution to ~\eqref{eq:fracheat} can also be explicitly written as
\[
  u(x,t) = \int_{\mathbb{R}^n} G_\alpha(x-y,t)u_0(y)dy,
\]
where $G_\alpha(x,t)$ is the Green's function with Fourier transform
$\hat{G}_{\alpha}(k,t)=\exp(-|\xi|^{\alpha}t)$.



\begin{figure}[htp]
 \begin{center}
  \includegraphics[totalheight=0.28\textheight]{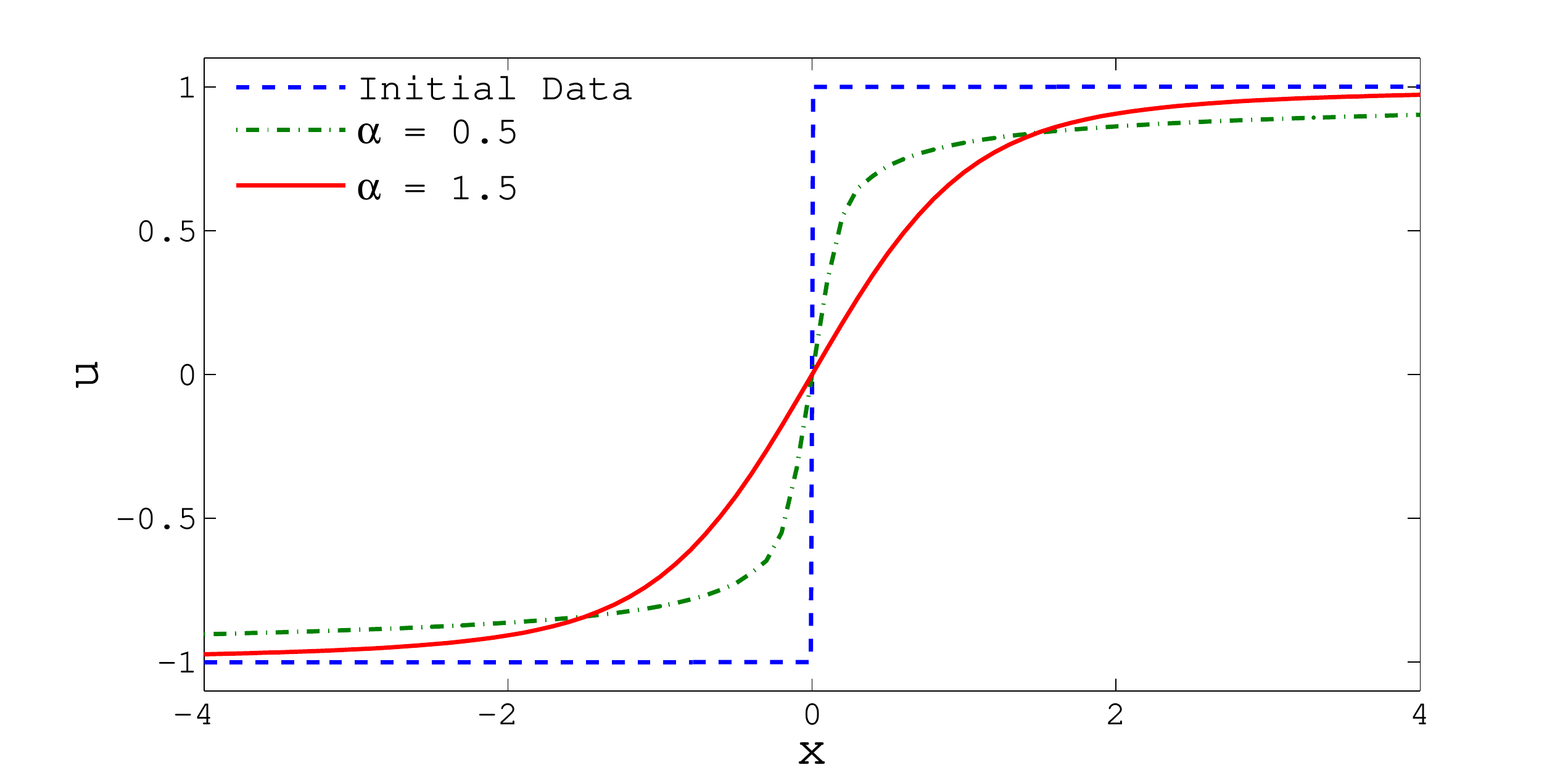}
 \end{center}
\caption{The spreading of the sign function under fractional diffusion
  equation~\eqref{eq:fracheat} at time $t=0.5$, for $\al=0.5$ and $\al=1.5$
  respectively.}
\label{fig:sngheat}
\end{figure}

Starting from the sign function $u_0(x)=\mbox{sign}(x)$ ,
the solution at $t=0.5$ is shown in Figure~\ref{fig:sngheat},
for $\alpha=0.5$ and $\alpha=1.5$, respectively.
The solution is computed with weight $w^{PER}$ on the domain $[-10,10]$, with grid size $h=0.1$ and
time step $\Delta t=0.01$ (for $\al=0.5$) and $\Delta t=0.0032$ (for $\al=1.5$). The treatment
of the far field boundary conditions $u(\pm \infty) = \pm  1$ is adapted from that in
Section~\ref{sec:truncation}, using the fact that $u(x)\pm 1 \sim |x|^{-\alpha}$
as $x\to \pm \infty$. While both solutions spread in space,
the larger $\al$ is, the smoother the solution becomes near its initial discontinuity, and
the closer the solution stays to its boundary condition.


\subsection{Fractal Burgers Equation}
In this subsection, several numerical solutions of the fractional Burgers
equation~\cite{MR1637513,MR2227237}
\begin{equation}\label{eq:fracBurgers}
     u_t + \partial_x\left(\frac{u^2}{2}\right) + \kappa(-\partial_{xx})^{\al/2}u=0
\end{equation}
are demonstrated. This equation (often called fractal Burgers equation in literature)
 is one of the most well-studied conservation laws with fractional Laplacian~\cite{MR1637513}.
 Many properties shared by general fractional  conservation laws can be understood through
 this simple equation. Depending on the initial  data and the exponent $\alpha$,
 general behaviours of solutions to~\eqref{eq:fracBurgers} are different.
 For increasing initial data with constant far field (say zero),
 the long time asymptotics is described by the inviscid Burgers equation for
 $\alpha \in (1,2]$~(see \cite{MR2377288}), but is governed by the fractional heat equation
 for $\alpha \in (0,1)$~(see \cite{MR2607346}). For decreasing initial data
 with constant far field and $\alpha \in (0,1)$, the solution can develop shocks under different conditions~\cite{MR2339805,MR1637513,MR2514389,MR2455893}.

Let $F_{j+1/2}^n$ be the numerical flux corresponding to $u^2/2$ at space $x_{j+1/2}$ and time
$t_n$, then a straightforward explicit discretization of~\eqref{eq:fracBurgers} is
\[
 \frac{u^{n+1}_j-u^n_j}{\Delta t} + \frac{F_{j+1/2}^n-F_{j-1/2}^n}{\Delta x}
+ \nu \sum_{k=-\infty}^\infty (u_j^n-u_k^n)w_{j-k}=0.
\]

%
%

For the initial condition $u_0(x)=\mbox{sign}(x)$, the solution
is a rarefaction wave in the inviscid case ($\kappa=0$), and is smoother in the
presence of fractional diffusion, as shown in Figure~\ref{fig:rareburgers}.
The computation domain is the interval $[-10,10]$ with grid size $h=0.1$ and
time step $0.1h^{\alpha}$, using the weights $w^{PER}$.
Despite the nonlinear term $(u^2/2)_x$ in~\eqref{eq:fracBurgers},
the solution has the same asymptotic far field boundary condition
$u(x)\pm 1 \sim |x|^{-\alpha}$ as $x\to \pm \infty$.

\begin{figure}[htp]
  \begin{center}
    \includegraphics[totalheight=0.25\textheight]{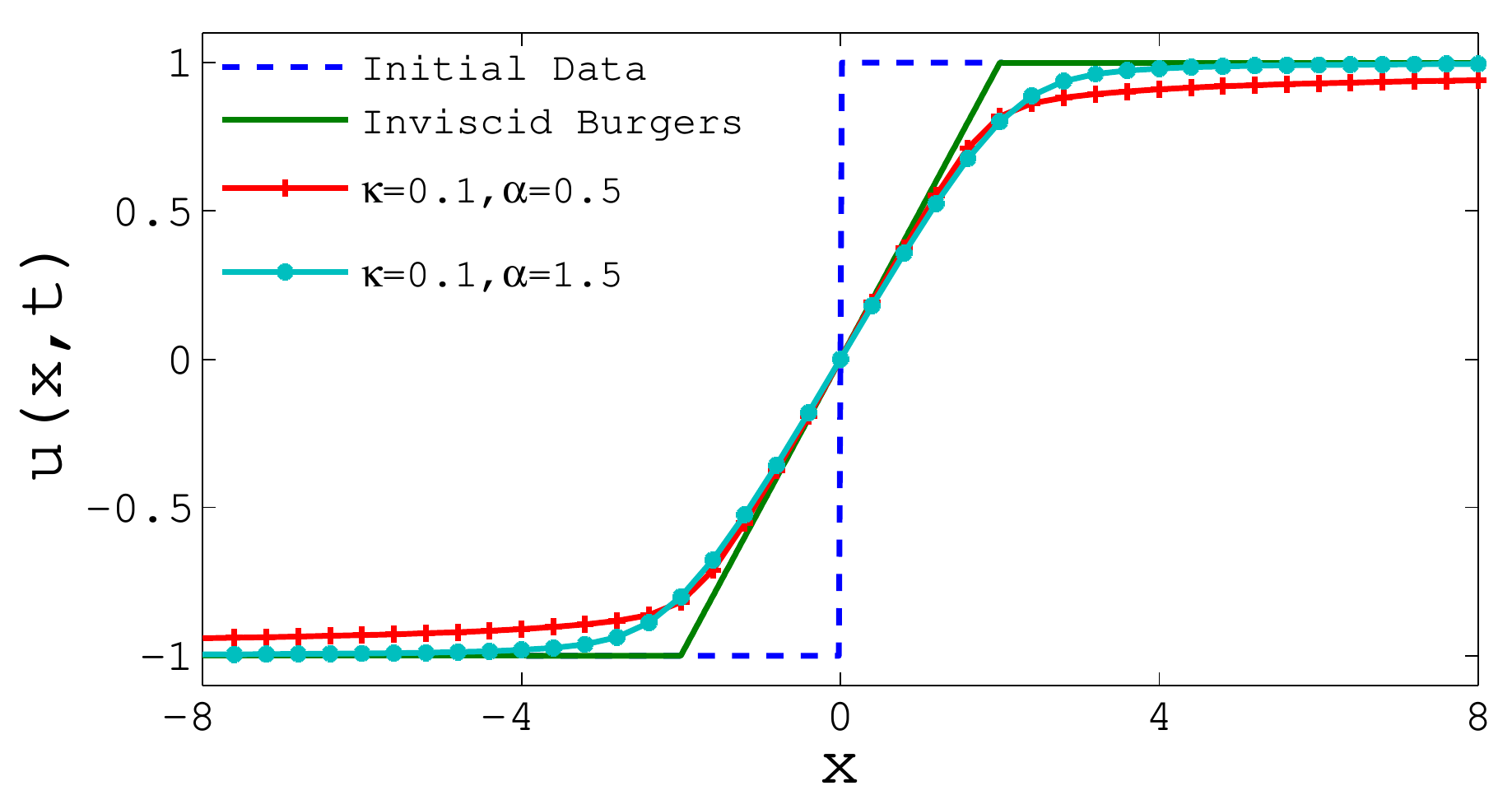}
  \end{center}
  \caption{The solution to the fractal Burgers equation~\eqref{eq:fracBurgers} at $t=2$, compared with
the rarefaction wave governed by the inviscid Burgers equation.}
\label{fig:rareburgers}
\end{figure}

When the initial data $u_0(x)=-\mbox{sign}(x)$ is chosen,
$u(x,t)=u_0(x)$ is a stationary shock in the inviscid case.
When the fractional diffusion is turned on, the solution may become
continuous when $\alpha$ is larger than one (see Figure~\ref{fig:Burgers_alpha})
or when the coefficient $\kappa$ is large enough (see Figure~\ref{fig:Burgers_kappa}).
Although these numerical evidences are not conclusive for the existence of the shocks,
especially when the profile could be poorly resolve, they are consistent
with limited existing results~\cite{MR2339805,MR1637513,MR2514389,MR2455893}.

\begin{figure}[htp]
    \begin{center}
        \includegraphics[totalheight=0.27\textheight]{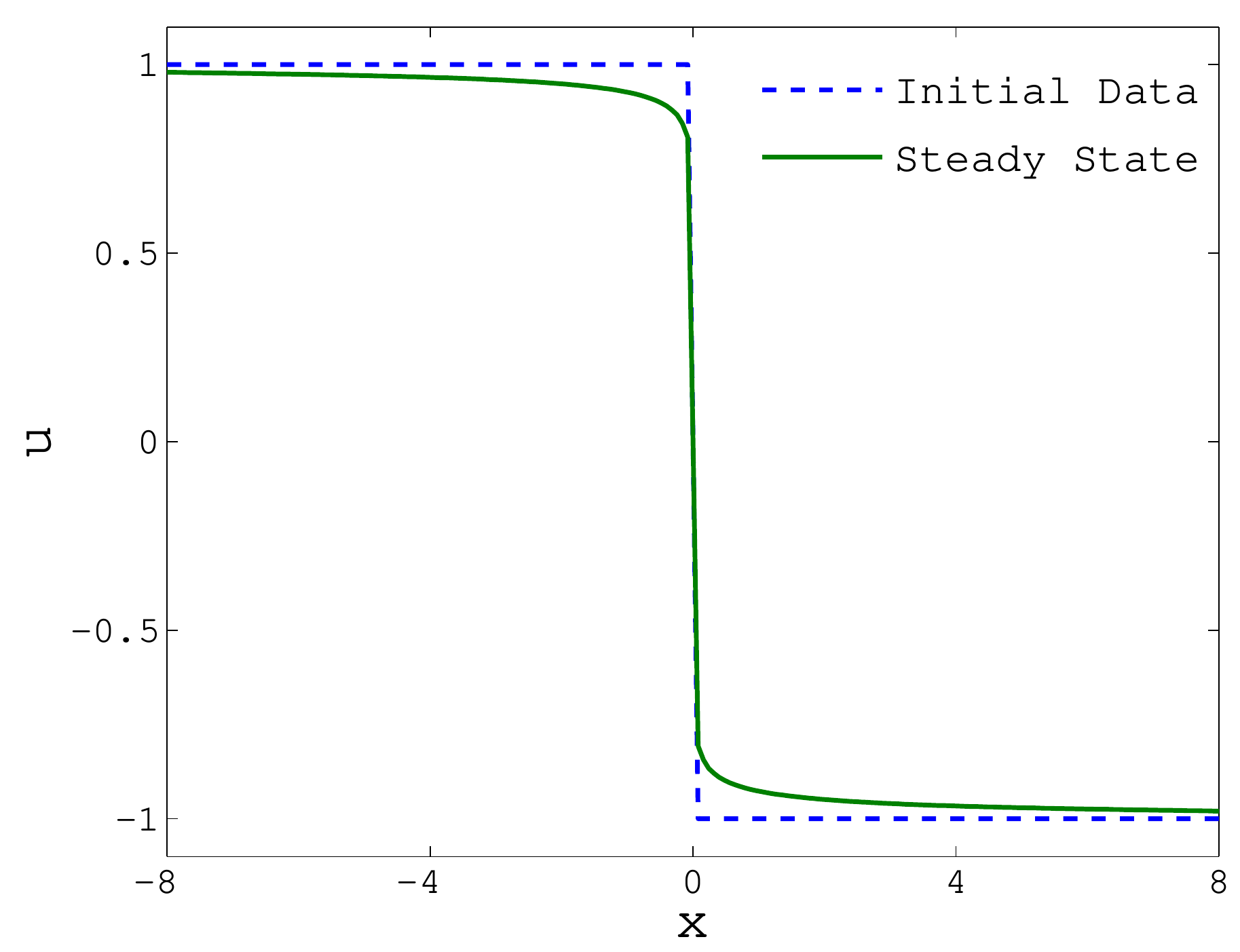} $~$
        \includegraphics[totalheight=0.27\textheight]{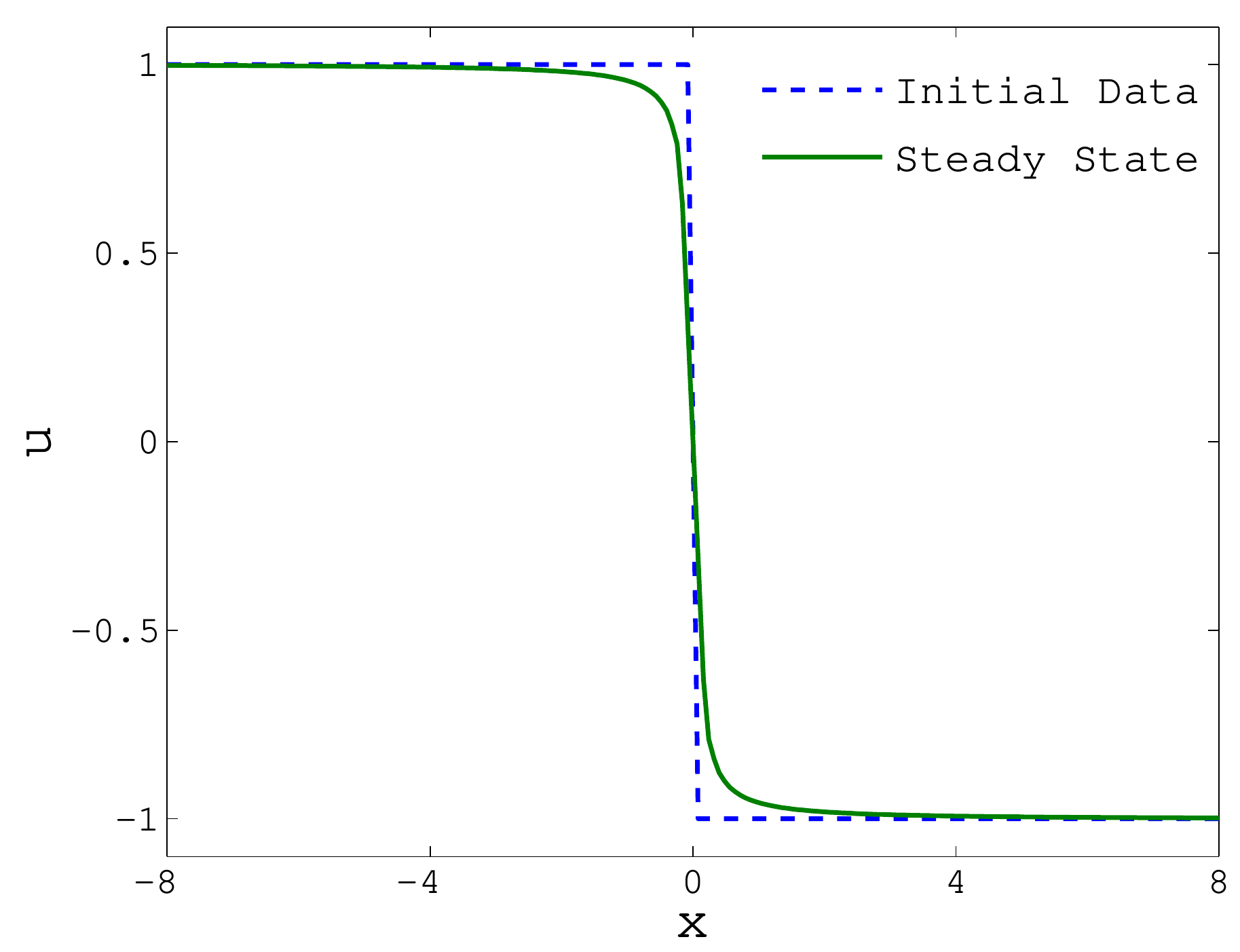}
    \end{center}
    \caption{The steady state of the fractal Burgers equation~\eqref{eq:fracBurgers} ($\kappa=1$)
    with initial data $u_0(x)=-\mbox{sign}(x)$ for $\al = 0.8$ (left) and $\al=1.2$ (right),
    respectively.}
    \label{fig:Burgers_alpha}
\end{figure}

\begin{figure}[htp]
    \begin{center}
        \includegraphics[totalheight=0.27\textheight]{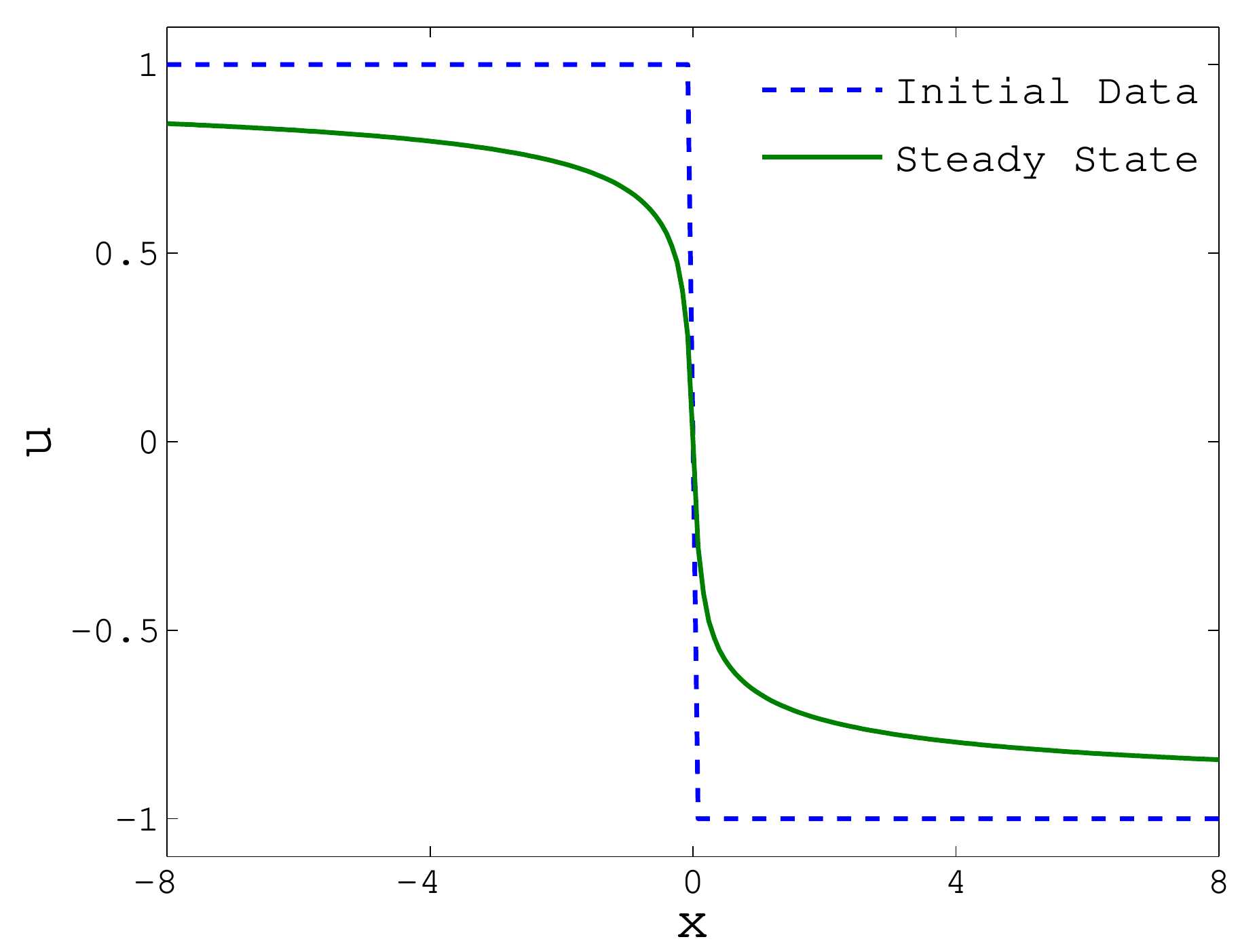} $~$
        \includegraphics[totalheight=0.27\textheight]{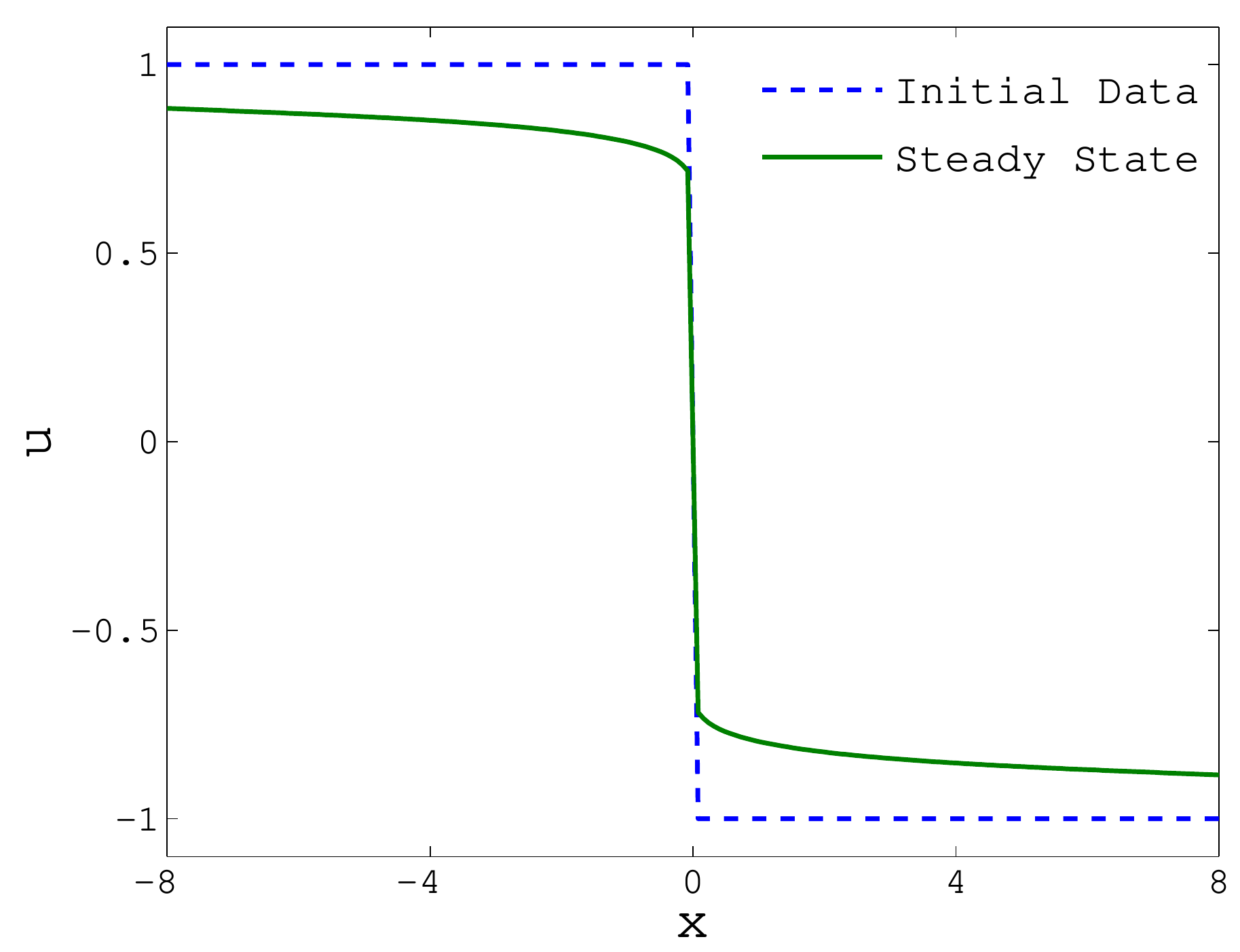}
    \end{center}
    \caption{The steady state of the fractal Burgers equation~\eqref{eq:fracBurgers}
    ($\al=0.4$) with initial data $u_0(x)=-\mbox{sign}(x)$ for $\kappa = 1.0$ (left)
    and $\kappa=0.1$ (right), respectively.
    The steady state seems continuous  when $\kappa$ is relatively large but becomes
``discontinuous'' when $\kappa$ is small.}
\label{fig:Burgers_kappa}
\end{figure}

For other integrable initial data with zero boundary condition at infinity, similar features
like smoothing of negative slopes and steepening of positive slopes are expected. For the
initial condition
\begin{equation}\label{eq:burginit}
    u_0(x) = \begin{cases}
        \cos x,\qquad &|x|\leq \pi/2,\cr
        0,& \mbox{otherwise},
    \end{cases}
\end{equation}
the solutions at $t=4$ and $t=8$ are shown in Figure~\ref{fig:sinsteep}.
When $\al<1$ and $\kappa$ is small, a discontinuity in the solution is expected~\cite{MR2339805},
although the precise conditions for the occurrence of  shocks are yet unknown.

\begin{figure}[htp]
    \begin{center}
        \includegraphics[totalheight=0.27\textheight]{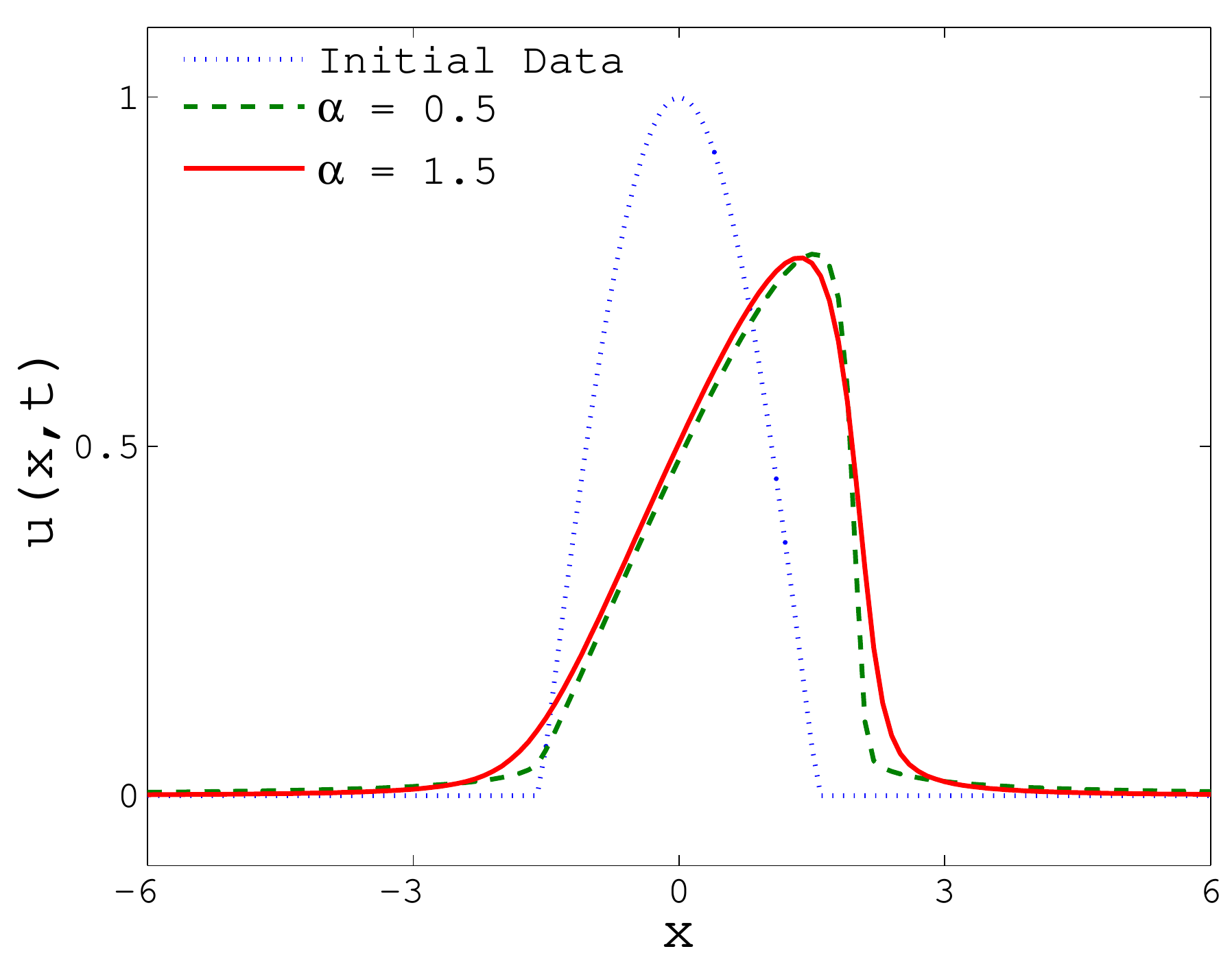} $~~$
        \includegraphics[totalheight=0.27\textheight]{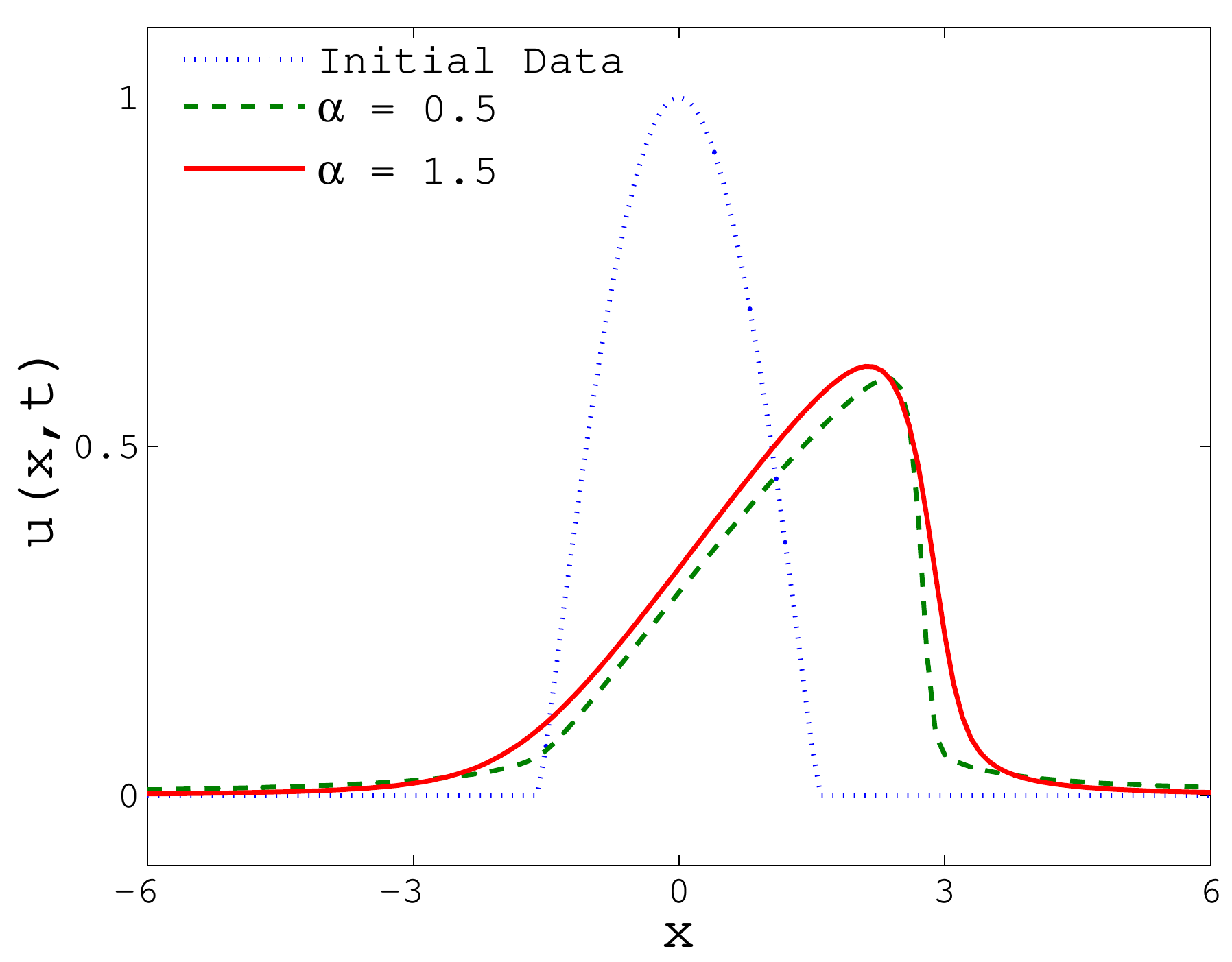}
    \end{center}
    \caption{The solution of the fractal Burgers equation at $t=4$ (left) and $t=8$ (right)
        with initial condition~\eqref{eq:burginit}.}
        \label{fig:sinsteep}
    \end{figure}

\subsection{Fractional thin film equation} The fractional thin film equation
$u_t = \nabla\cdot\big(u\nabla (-\Delta)^{\alpha/2}u\big)$ appears in fracture
dynamics~\cite{imbert2014self}. Below we focus on the equation in similarity variables, that is,
\begin{equation}\label{eq:fracthinfilm}
 u_t = \nabla \cdot \big( u\nabla (-\Delta)^{\alpha/2}u\big) + \lambda
\nabla\cdot(xu),
\end{equation}
which possesses a stationary steady instead of spreading and decaying densities.
If we choose $\lambda = 2(1+\alpha)C_{1,\alpha}$ and initial data with the total conserved mass
\[
  M=\int_\mathbb{R} u(x,t) \mathrm{d}x = \sqrt{\pi}\Gamma\left(2+\frac{a}{2}\right)
  \Gamma\left(\frac{5+a}{2}\right)^{-1},
\]
the it is each to check that
$u_\infty(x) = (1-x^2)^{1+\alpha/2}$ is the steady state. The
time evolution of the solution starting from the initial data
\begin{equation}\label{eq:thinfilmu0}
  u_0(x) = \frac{M}{\sqrt{\pi}}\left( 0.8e^{-4(x-1)^2}+1.6e^{-16(x+2)^2}\right)
\end{equation}
is shown in Figure~\ref{fig:fracthinfilm}. Since the solution is essentially
supported on an interval, the computational domain $[-4,4]$ is chosen, with
grid size $h=0.01$ and time step $\Delta t = 0.0001$. Clearly the steady state
$u_\infty(x)$ is approached as time evolves.

\begin{figure}[htp]
 \begin{center}
  \includegraphics[totalheight=0.27\textheight]{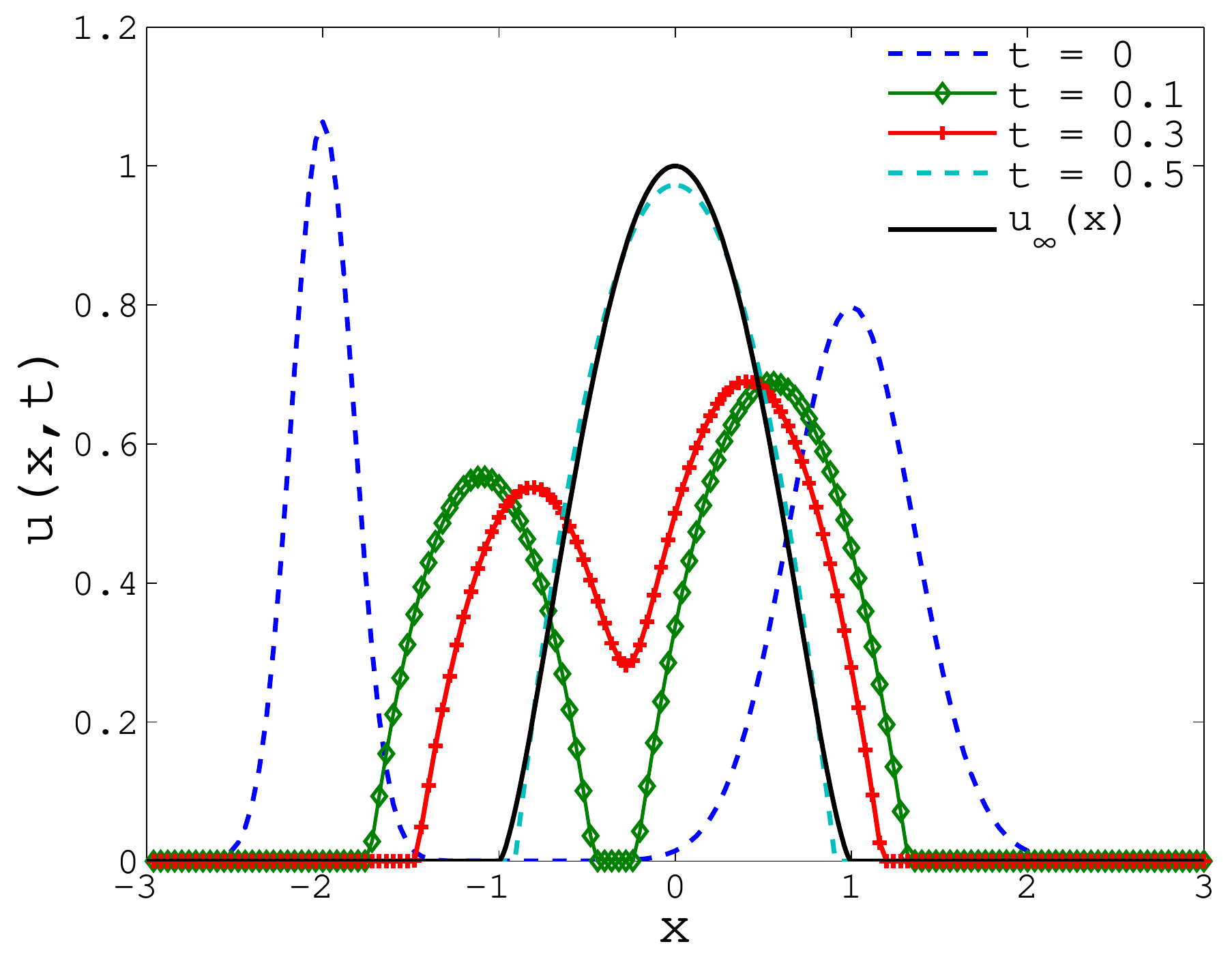}
 \end{center}
 \caption{The evolution of solutions to~\eqref{eq:fracthinfilm}
 start with initial data~\eqref{eq:thinfilmu0} consisting two Gaussians.}
\label{fig:fracthinfilm}
\end{figure}

\section{Conclusions}

In this article we performed a comprehensive study of finite difference
approximation of the fractional Laplacian operator of the form~\eqref{FLh}.
The resulting discrete operator is a multiplier in the spectral space,
and the connection between the weights and the rescaled symbol is established.
Many schemes are derived or reviewed in this context, and are compared
in various ways about the asymptotic scaling of the weights and their
order of accuracy. Other practical issues are also highlighted, like the treatment
of the far field boundary conditions and the effect the non-smoothness of the solutions
on the resulting accuracy. The schemes can be applied to many PDEs with fractional Laplacian,
and provide  robust tools for numerical investigation on these
more difficult equations besides theoretical analysis.

This numerical scheme can be generalised easily to translation-invariant operator
$\mathcal{L}$ with symbol or multiplier $\tilde{M}$ in the spectral space. That is,
\begin{equation}\label{eq:lscheme}
  \mathcal{L}u_j =  \sum_{k=-\infty}^\infty w_ku_{j-k},
\end{equation}
such that
\[
  \tilde{M}_h(\xi):= h \sum_{k=-\infty}^\infty w_ke^{-i\xi x_k} \sim \tilde{M}(\xi),
\]
near the origin. From a given symbol $\tilde{M}_h(\xi)$, the weights are expressed as
\[
  w_k = \frac{1}{2\pi}\int_{-\pi/h}^{\pi/h} \tilde{M}_h(\xi)e^{i\xi x_k} \mathrm{d}\xi.
\]
For operators like $\sqrt{1-\Delta}$, the weights $w_k$ normally depend
on the grid size $h$ in a complicated way, but for scale invariant operators
like $(-\Delta)^{\alpha/2}$ and $(-\Delta)^{-\alpha/2}$, $h$ can be factored out
from the weights, as we did in Section~\ref{sec:specrep}. Further more, if $\tilde{M}$
vanishes at the origin, it is reasonable to require that
$\tilde{M}_h(0)=0$ and the scheme~\eqref{eq:lscheme} can be written as~\eqref{FLh}.

While the framework of the scheme~\eqref{FLh} can be extended in a straightforward
way into higher dimensions, several practical limitations do appear.
Explicit expressions of the weights like~\eqref{eq:lommelsp} and~\eqref{eq:expPERweights}
are less likely for any given symbol like $M(\xi) = (\xi_1^2+\xi_2^2+\cdots+\xi_n^2)^{\alpha/2}$.
As a result, numerical quadratures of oscillatory integrals are inevitable~\cite{MR2211043,MR2068828}.
The far field boundary conditions are also more difficult to treat.
In one dimension, the function only decays in two directions. But the flat tails
in higher dimensions, if they exist, are much more complicated. As a result, there are
many important future questions to be answered for the fractional Laplacian
in particular and for nonlocal equations in general.

\end{document}